\documentclass[12pt, letterpaper,reqno]{amsart}
\usepackage{amsthm,amssymb,mathrsfs}

\usepackage{psfrag, graphicx} 

\DeclareFontFamily{OT1}{ptm}{}
\DeclareFontShape{OT1}{ptm}{m}{n} { <-> ptmr}{}
\DeclareFontShape{OT1}{ptm}{m}{it}{ <-> ptmri}{}
\DeclareFontShape{OT1}{ptm}{m}{sl}{ <->ptmro}{}
\DeclareFontShape{OT1}{ptm}{m}{sc}{ <-> ptmrc}{}
\DeclareFontShape{OT1}{ptm}{b}{n} { <-> ptmb}{}
\DeclareFontShape{OT1}{ptm}{b}{it}{ <-> ptmbi}{}     
\DeclareFontShape{OT1}{ptm}{bx}{n} {<->ssub * ptm/b/n}{}
\DeclareFontShape{OT1}{ptm}{bx}{it}{<->ssub * ptm/b/it}{}

\DeclareSymbolFont{bold}{OT1}{ptm}{b}{n}
\DeclareMathAlphabet{\mathbf}{OT1}{ptm}{b}{n}  
\DeclareMathAlphabet{\mathrm}{OT1}{ptm}{m}{n}

\DeclareFontFamily{OT1}{psy}{}      
\DeclareFontShape{OT1}{psy}{m}{n}{ <-> s * [0.9] psyr}{}
\DeclareFontFamily{OMS}{ptm}{}     
\DeclareFontShape{OMS}{ptm}{m}{n}{ <8> <9> <10> gen * cmsy }{}
\DeclareFontFamily{OMS}{cmtt}{}     
\DeclareFontShape{OMS}{cmtt}{m}{n}{ <8> <9> <10> gen * cmsy }{}

\SetSymbolFont{operators}{normal}{OT1}{ptm}{m}{n}   
\SetSymbolFont{operators}{bold}{OT1}{ptm}{b}{n}     
\DeclareSymbolFont{emsy}{OT1}{ptm}{m}{it}
\DeclareSymbolFont{emsr}{OT1}{ptm}{m}{n}
\DeclareSymbolFont{emcmr}{OT1}{cmr}{m}{n}   
\DeclareSymbolFont{emsymb}{OT1}{psy}{m}{n}  
\DeclareMathSymbol a{\mathalpha}{emsy}{"61}
\DeclareMathSymbol b{\mathalpha}{emsy}{"62}
\DeclareMathSymbol c{\mathalpha}{emsy}{"63}
\DeclareMathSymbol d{\mathalpha}{emsy}{"64}
\DeclareMathSymbol e{\mathalpha}{emsy}{"65}
\DeclareMathSymbol f{\mathalpha}{emsy}{"66}
\DeclareMathSymbol g{\mathalpha}{emsy}{"67}
\DeclareMathSymbol h{\mathalpha}{emsy}{"68}
\DeclareMathSymbol i{\mathalpha}{emsy}{"69}
\DeclareMathSymbol j{\mathalpha}{emsy}{"6A}
\DeclareMathSymbol k{\mathalpha}{emsy}{"6B}
\DeclareMathSymbol l{\mathalpha}{emsy}{"6C}
\DeclareMathSymbol m{\mathalpha}{emsy}{"6D}
\DeclareMathSymbol n{\mathalpha}{emsy}{"6E}
\DeclareMathSymbol o{\mathalpha}{emsy}{"6F}
\DeclareMathSymbol p{\mathalpha}{emsy}{"70}
\DeclareMathSymbol q{\mathalpha}{emsy}{"71}
\DeclareMathSymbol r{\mathalpha}{emsy}{"72}
\DeclareMathSymbol s{\mathalpha}{emsy}{"73}
\DeclareMathSymbol t{\mathalpha}{emsy}{"74}
\DeclareMathSymbol u{\mathalpha}{emsy}{"75}
\DeclareMathSymbol v{\mathalpha}{emsy}{"76}
\DeclareMathSymbol w{\mathalpha}{emsy}{"77}
\DeclareMathSymbol x{\mathalpha}{emsy}{"78}
\DeclareMathSymbol y{\mathalpha}{emsy}{"79}
\DeclareMathSymbol z{\mathalpha}{emsy}{"7A}
\DeclareMathSymbol A{\mathalpha}{emsy}{"41}
\DeclareMathSymbol B{\mathalpha}{emsy}{"42}
\DeclareMathSymbol C{\mathalpha}{emsy}{"43}
\DeclareMathSymbol D{\mathalpha}{emsy}{"44}
\DeclareMathSymbol E{\mathalpha}{emsy}{"45}
\DeclareMathSymbol F{\mathalpha}{emsy}{"46}
\DeclareMathSymbol G{\mathalpha}{emsy}{"47}
\DeclareMathSymbol H{\mathalpha}{emsy}{"48}
\DeclareMathSymbol I{\mathalpha}{emsy}{"49}
\DeclareMathSymbol J{\mathalpha}{emsy}{"4A}
\DeclareMathSymbol K{\mathalpha}{emsy}{"4B}
\DeclareMathSymbol L{\mathalpha}{emsy}{"4C}
\DeclareMathSymbol M{\mathalpha}{emsy}{"4D}
\DeclareMathSymbol N{\mathalpha}{emsy}{"4E}
\DeclareMathSymbol O{\mathalpha}{emsy}{"4F}
\DeclareMathSymbol P{\mathalpha}{emsy}{"50}
\DeclareMathSymbol Q{\mathalpha}{emsy}{"51}
\DeclareMathSymbol R{\mathalpha}{emsy}{"52}
\DeclareMathSymbol S{\mathalpha}{emsy}{"53}
\DeclareMathSymbol T{\mathalpha}{emsy}{"54}
\DeclareMathSymbol U{\mathalpha}{emsy}{"55}
\DeclareMathSymbol V{\mathalpha}{emsy}{"56}
\DeclareMathSymbol W{\mathalpha}{emsy}{"57}
\DeclareMathSymbol X{\mathalpha}{emsy}{"58}
\DeclareMathSymbol Y{\mathalpha}{emsy}{"59}
\DeclareMathSymbol Z{\mathalpha}{emsy}{"5A}
\DeclareMathSymbol{\bullet}{\mathalpha}{emsymb}{"B7}
\DeclareMathSymbol{\regis}{\mathalpha}{emsymb}{"D2}
\def\Bullet{\leavevmode\unkern{$\m@th\bullet$}\kern.32em\ignorespaces}
\def\Regis{\leavevmode\raise.5ex\hbox{$\m@th\regis$}}
\DeclareMathSymbol +{\mathbin}{emcmr}{`+}
\DeclareMathSymbol ={\mathrel}{emcmr}{`=}  
\DeclareMathSymbol{\Gamma}{\mathalpha}{emcmr}{"00}
\DeclareMathSymbol{\Delta}{\mathalpha}{emcmr}{"01}
\DeclareMathSymbol{\Theta}{\mathalpha}{emcmr}{"02}
\DeclareMathSymbol{\Lambda}{\mathalpha}{emcmr}{"03}
\DeclareMathSymbol{\Xi}{\mathalpha}{emcmr}{"04}
\DeclareMathSymbol{\Pi}{\mathalpha}{emcmr}{"05}
\DeclareMathSymbol{\Sigma}{\mathalpha}{emcmr}{"06}
\DeclareMathSymbol{\Upsilon}{\mathalpha}{emcmr}{"07}
\DeclareMathSymbol{\Phi}{\mathalpha}{emcmr}{"08}
\DeclareMathSymbol{\Psi}{\mathalpha}{emcmr}{"09}
\DeclareMathSymbol{\Omega}{\mathalpha}{emcmr}{"0A}
\DeclareMathSizes{7.6}{8}{6}{5}

\theoremstyle{plain}
\newtheorem{theorem}[subsection]{Theorem}
\newtheorem{proposition}[subsection]{Proposition}
\newtheorem{lemma}[subsection]{Lemma}
\newtheorem{corollary}[subsection]{Corollary}
\newtheorem*{transthm}{Transversality Theorem}

\theoremstyle{definition}
\newtheorem{definition}[subsection]{Definition}
\newtheorem{remark}[subsection]{Remark}
\newtheorem{remarks}[subsection]{Remarks}
\newtheorem{remcon}[subsection]{Remark/Conjecture}
\newtheorem{example}[subsection]{Example}
\newtheorem{noname}[subsection]{}

\newcommand{\co}{\colon \thinspace}

\newcommand{\R}{{\mathbb R}}
\newcommand{\Z}{{\mathbb Z}}
\newcommand{\la}{\lambda}

\newcommand{\abs}[1]{\left| {#1} \right|}
\renewcommand{\leq}{\leqslant}
\renewcommand{\geq}{\geqslant}
\renewcommand{\epsilon}{\varepsilon}
\DeclareMathOperator{\area}{Area}
\newcommand{\N}{{\mathbb N}}

\newcommand{\Fs}{{\mathscr F}}
\newcommand{\Fx}{{\mathscr F}_{\!x}}
\newcommand{\Fy}{{\mathscr F}_{\!y}}
\newcommand{\Fz}{{\mathscr F}_{\!z}}
\newcommand{\ellx}{\ell_{\! x}}
\newcommand{\elly}{\ell_{\! y}}

\renewcommand{\preceq}{\preccurlyeq}
\renewcommand{\succeq}{\succcurlyeq}
\newcommand{\IP}{{\rm{IP}}} 
\newcommand{\IPk}{{\rm{IP}}^{(k)}} 
\DeclareMathOperator{\FVol}{FVol}
\DeclareMathOperator{\Vol}{Vol}
\DeclareMathOperator{\RVol}{RVol}
\DeclareMathOperator{\RArea}{RArea}
\DeclareMathOperator{\id}{id}

\DeclareMathOperator{\interior}{int}
\DeclareMathOperator{\closure}{closure}

\makeatletter
  \def\tagform@#1{\maketag@@@{%
   \textbf{(\ignorespaces#1\unskip\@@italiccorr)}}}%
   \renewcommand{\eqref}[1]{\textup{\maketag@@@{(\ignorespaces%
        {\ref{#1}}\unskip\@@italiccorr)}}}
\makeatother

\begin{document}

\title[Density of isoperimetric spectra]{Density of isoperimetric
spectra} 

\author{Noel Brady}
\address{Mathematics Department\\
        University of Oklahoma\\
        Norman, OK 73019\\
        USA}
\email{nbrady@math.ou.edu\\forester@math.ou.edu}

\author{Max Forester}

\begin{abstract} 
We show that the set of $k$-dimensional isoperimetric exponents of
finitely presented groups is dense in the interval $[1, \infty)$ for $k
\geq 2$. Hence there is no higher-dimensional analogue of Gromov's gap
$(1,2)$ in the isoperimetric spectrum. 
\end{abstract}

\thanks{Partially supported by NSF grants \ DMS-0505707 (Brady) and
\ DMS-0605137 (Forester).}  

\maketitle

\thispagestyle{empty}

\vspace{-.1in}
\centerline{\Small \textit{Dedicated to the memory of John Stallings}}
\smallskip

\section{Introduction}
Dehn functions of groups have been the subject of intense activity over the 
past two decades. 
The Dehn function $\delta(x)$ of a group $G$ is a quasi-isometry
invariant which describes the best possible isoperimetric inequality that
holds in any geometric model for the group. Specifically, for a given
$x$, $\delta(x)$ is the smallest number $A$ such that every
null-homotopic loop of length at most $x$ bounds a disk of area $A$ or
less. One defines length and area combinatorially, based on a
presentation $2$-complex for $G$, and the resulting Dehn function is well
defined up to coarse Lipschitz equivalence. If $G$ is the fundamental
group of a closed Riemannian manifold $M$, then ordinary length and area
in $M$ may be used instead, and one obtains an equivalent function. (This
seemingly modest but non-trivial result is sometimes called the Filling
Theorem; see \cite{mrb:bfs} or \cite{BuTa} for a proof.) 

Due in large part to the work of Birget, Rips, and
Sapir \cite{BRS} we now have a fairly complete understanding of which
functions are Dehn functions of finitely presented groups. In the case of
power functions, one defines the {\em isoperimetric spectrum} to be the
following (countable) subset of the line: 
\[ \IP \ = \ \{ \, \alpha \in [1,\infty) \mid f(x)=x^\alpha \text{ is
  equivalent to a Dehn function} \,\}. \] 
Combining results of \cite{gromov1,BrBr,bbfs,BRS}, we know 
that the isoperimetric spectrum has closure $\{1\} \cup [2,\infty)$ and
that it contains all rational numbers in $[2,\infty)$. Moreover, in the
range $(4, \infty)$, it contains (almost exactly) those numbers having
computational complexity below a certain threshold \cite{BRS}. The gap
$(1,2)$ reflects Gromov's theorem to the effect that every finitely
presented group with sub-quadratic Dehn function is hyperbolic, and hence
has linear Dehn function. Several proofs of this result are known: see
\cite{gromov1,olshanskii,papa2,bowditch}. 

By analogy with ordinary Dehn functions, one defines the $k$-dimensional
Dehn function $\delta^{(k)}(x)$, describing the optimal $k$-dimensional
isoperimetric inequality that holds in $G$. Given $x$,
$\delta^{(k)}(x)$ is the smallest $V$ such that every $k$-dimensional
sphere of volume at most $x$ bounds a $(k+1)$-dimensional ball of volume
$V$ or less. One uses combinatorial notions of volume, based on a chosen
$k$-connected model for $G$. Again, up to coarse Lipschitz equivalence,
$\delta^{(k)}(x)$ is preserved by quasi-isometries \cite{AWP}, and in
particular does not depend on the choice of model for $G$. 

Precise details regarding the definition of $\delta^{(k)}(x)$ are given
in Section \ref{prelimsection}. Nevertheless, it is worth emphasizing
here that we are filling spheres with balls, which is quite different from
filling spheres with chains, or cycles with chains (the latter of which
leads to the \emph{homological Dehn function}). It turns out that we do 
indeed need to make use of other variants (namely, the \emph{strong Dehn
function} -- see Section \ref{prelimsection}), but for us the primary
object of most immediate geometric interest is the Dehn function as
described above. 

In this paper we are concerned with the following question: what is the
possible isoperimetric behavior of groups, in various dimensions? 
For each positive integer $k$ one defines the {\em $k$-dimensional 
isoperimetric spectrum}: 
\[ \IPk \ = \ \{ \, \alpha \in [1,\infty) \mid f(x)=x^\alpha \text{ is
equivalent to a $k$-dimensional Dehn function} \,\}. \]
Until recently, relatively little was known about $\IPk$, especially when
$k \geq 3$. A few results concerning $\IP^{(2)}$ were known: in
\cite{alonso+,WP,wang} it was shown that $\IP^{(2)}$ contains infinitely
many points in the interval $[3/2,2)$, and various lower and upper bounds
were located throughout $[2, \infty)$; also in \cite{BrBr,mb:plms} it was
shown that $\IP^{(2)} \cap [3/2,2)$ is dense in $[3/2,2)$ and that $2,3
\in \IP^{(2)}$. 

The recent paper \cite{bbfs} established that $\IPk$ is dense in $[1 +
\frac{1}{k}, \infty)$ and contains all rational numbers in this range. 
The endpoint $1 + \frac{1}{k}$ corresponds to the isoperimetric
inequality represented by spheres in Euclidean space. The main purpose of
the present paper is to address the sub-Euclidean range $(1, 1 +
\frac{1}{k})$ and establish the existence of isoperimetric exponents
throughout this interval, for $k \geq 2$. 

To state our results we need some notation. 
If $A$ is a non-singular $n \times n$ integer matrix, let $G_A$ 
denote the ascending HNN extension of $\Z^n$ with monodromy $A$. 
Our first result is the following. 

\begin{theorem}\label{mainthm1}
Let $A$ be a $2 \times 2$ integer matrix with eigenvalues
$\lambda, \mu$ such that $\lambda > 1 > \mu$ and $\lambda\mu>1$. 
Then the $2$-dimensional Dehn function of $G_A$ is equivalent to 
$x^{2 + \log_{\lambda}(\mu)}$. 
\end{theorem}

In Section \ref{densitysect} we show that the exponents arising in
the theorem are dense in the interval $(1,2)$. Thus, roughly half of
these groups have sub-Euclidean filling volume 
for $2$-spheres, occupying densely the desired range of possible behavior. 

Given an $n \times n$ matrix $A$, the \emph{suspension}  $\Sigma A$ of
$A$ is the $(n+1)\times (n+1)$ matrix obtained by direct sum with the
$1\times 1$ identity matrix. Since $G_{\Sigma A} \cong G_A \times \Z$,
results from \cite{bbfs} imply the following (see Section
\ref{highersect} for details). 

\begin{theorem}\label{mainthm2} 
Let $G_A$ be as in Theorem \ref{mainthm1}. Then the $(i+2)$-dimensional
Dehn function of $G_{\Sigma^i A}$ is equivalent to $x^s$ where $s =
\frac{(i + 1)\alpha - i}{i\alpha - (i -1)}$ and $\alpha = 2 +
\log_{\lambda}(\mu)$. 
\end{theorem}

Given that the numbers $\alpha$ are dense in the interval $(1,2)$, it
follows that the exponents $s$ are dense in $(1, (i+2)/(i+1))$. Together
with Corollary E of \cite{bbfs}, we have the following result,
illustrated in Figure \ref{fig:spectra}. 

\begin{corollary}
$\IPk$ is dense in $[1, \infty)$ for $k \geq 2$. 
\end{corollary}

\begin{figure}[ht]
	\begin{center}
		\psfrag{1}{{\scriptsize $1$}}
		\psfrag{2}{{\scriptsize $2$}}
		\psfrag{3}{{\scriptsize $3$}}
		\psfrag{32}{{\scriptsize $\frac{3}{2}$}}
		\psfrag{43}{{\scriptsize $\frac{4}{3}$}}
		\psfrag{54}{{\scriptsize $\frac{5}{4}$}}
		\psfrag{65}{{\scriptsize $\frac{6}{5}$}}
		\psfrag{i1}{{\scriptsize $\IP$}}
		\psfrag{i2}{{\scriptsize $\IP^{(2)}$}}
		\psfrag{i3}{{\scriptsize $\IP^{(3)}$}}
		\psfrag{i4}{{\scriptsize $\IP^{(4)}$}}
		\psfrag{i5}{{\scriptsize $\IP^{(5)}$}}
                \psfrag{ga}{{\scriptsize $G_A$}}
                \psfrag{ga1}{{\scriptsize $G_{\Sigma A}$}}
                \psfrag{ga2}{{\scriptsize $G_{\Sigma^2 A}$}}
                \psfrag{ga3}{{\scriptsize $G_{\Sigma^3 A}$}}
		\includegraphics[width=4in]{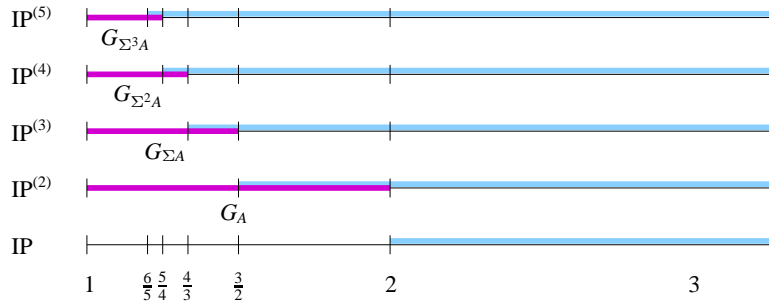}
	\end{center}
	\caption{Isoperimetric exponents of $G_{\Sigma^i A}$. 
The blue intervals indicate isoperimetric exponents
for the groups constructed in \cite{bbfs}. 
}\label{fig:spectra}  
\end{figure}

\subsection*{Methods}
The methods used here to establish isoperimetric inequalities for $G_A$
are quite different from those used in \cite{bbfs}. In the latter work, 
a slicing argument was used to estimate volume based on information coming
from one-dimensional Dehn functions. This approach is rather less promising in
the sub-Euclidean realm, since there are no one-dimensional Dehn
functions there to reduce to. (Reducing to larger Dehn functions does
not seem feasible, at least by similar methods.) 

Instead we must find and measure least-volume fillings of $2$-spheres in
$G_A$ directly, using properties of the particular geometry of this
group. We work with a piecewise Riemannian cell complex with a metric
locally modeled on a solvable Lie group $\R^2 \rtimes \R$. This metric is
particularly simple from the point of view of the given coordinates, and
these preferred coordinates make possible various volume and area
calculations that are central to our arguments. 

The preferred coordinates just mentioned do not behave well
combinatorially, however. Coordinate lines pass through cells in an
aperiodic manner, and this cannot be remedied by simply changing the cell
structure. If one attempts to measure volume combinatorially, counting
cells by passing between cells and their neighbors in an organized fashion
(as with ``$t$-corridor'' arguments, for example), one loses the
advantage of the preferred coordinates conferred by the special geometry
of these groups. To count cells, therefore, we use integration and divide
by the volume of a cell. 

The combinatorial structure is still relevant, however. The piecewise
Riemannian model is not a manifold, and its branching behavior is a
prominent feature of the geometry of $G_A$. In order to make clean
transitions between the combinatorial and Riemannian viewpoints, we use
the transversality technology of Buoncristiano, Rourke, and 
Sanderson \cite{BuRoSa}. This provides the appropriate notion of van
Kampen diagrams for higher-dimensional spheres and fillings. 
Transversality also helps in dealing with
singular maps, which otherwise present technical difficulties. 

One other technical matter deserves mention: in order to apply results of
\cite{bbfs} to deduce Theorem \ref{mainthm2}, we are obliged to find
bounds for the \emph{strong Dehn function}, which encodes uniform
isoperimetric inequalities for fillings of surfaces by arbitrary
$3$-manifolds. See Section \ref{prelimsection} for definitions and results
concerning the strong Dehn function. 

\begin{remcon}
The groups $G_A$ in Theorem \ref{mainthm1} were classified up to
quasi-isometry by Farb and Mosher \cite{fm3}. At the time, none of the
usual quasi-isometry invariants could distinguish these groups, but
the two-dimensional Dehn function apparently does so quite well. 
We conjecture that it is in fact a complete invariant for this class of
groups. What is missing is the knowledge that the real number
$\log_{\lambda}(\mu)$ determines the diagonal matrix $\big( \,
{}^{\lambda}_0 \ {}^0_{\mu} \, \big)$ up to a rational power. One needs
to take into account the specific assumptions on the integer matrix $A$
(eg. having a contracting eigenspace), to rule out examples
such as $\big( \, {}^{4}_0 \ {}^0_{2} \, \big)$ and $\big( \, {}^{9}_0 \
{}^0_{3} \, \big)$. 
\end{remcon}

\section{Preliminaries}\label{prelimsection}

In this section we discuss in detail some of the key notions needed to
carry out the proofs of the theorems. First we give a brief account
of the transversality theory of Buoncristiano, Rourke, and
Sanderson. Then we discuss volume, Dehn functions of various types, and
some basic results concerning these.

\subsection*{Handles and transverse maps} Using transversality, a map
from a manifold to a cell complex can be put into a nice form, called a
transverse map \cite{BuRoSa}. Transverse maps induce 
\emph{generalized handle decomopositions} of manifolds, which will play
the role of van Kampen diagrams in higher dimensions. 
Whereas \emph{admissible maps} were used for this purpose in \cite{bbfs},
transverse maps have additional structure, incorporating combinatorial
information dependent on the way cells meet locally in the target complex. 

An \emph{index $i$ handle} (or \emph{generalized handle}) of dimension
$n$ is a product $\Sigma^i \times D^{n-i}$, where $\Sigma^i$ is a
compact, connected $i$-dimensional manifold with boundary, and $D^{n-i}$
is a closed disk. Let $M$ be a closed $n$-manifold. A \emph{generalized
handle decomposition} of $M$ is a filtration $\emptyset = M^{(-1)}
\subset M^{(0)} \subset \cdots \subset M^{(n)} = M$ by codimension-zero
submanifolds, such that for each $i$, $M^{(i)}$ is obtained from
$M^{(i-1)}$ by attaching finitely many index $i$ handles, 
as follows. To attach a single handle $H = \Sigma^i \times D^{n-i}$,
choose an embedding $h \co \partial \Sigma^i \times D^{n-i} 
\to \partial M^{(i-1)}$ and form the manifold $M^{(i-1)} \cup_h H$. Note that
handle attachment is always along $\partial \Sigma^i \times D^{n-i}$, and
never along $\Sigma^i \times \partial 
D^{n-i}$. To attach several handles, we require that the
attaching maps have disjoint images in $\partial M^{(i-1)}$, so that
the order of attachment does not matter. Note that both $M^{(i-1)}$ and the
individual handles $H$ are embedded in $M^{(i)}$. 

If every $\Sigma^i$ is a disk then this is the usual notion of handle
decomposition arising in classical Morse theory. Some new things can
occur by varying 
$\Sigma^i$, however. For instance, we allow $\Sigma^i$ to be closed, in
which case the attaching map is empty and $M^{(i-1)} \cup_h H$ is the
disjoint union $M^{(i-1)} \sqcup H$. Such a handle is called a
\emph{floating handle}. For example, $M^{(0)}$ is formed from $M^{(-1)} =
\emptyset$ by attaching (floating) $0$-handles $D^0 \times D^n$, and
$M^{(0)}$ is simply several copies of $D^n$. (The lowest-index handles
will always be floating ones.) Another phenomenon is that
handles may be embedded in $M$ in topologically interesting ways, as in
the following example. 

\begin{example}
Given a closed orientable $3$-manifold $M$, we may construct a
generalized handle decomposition as follows. 
Let $K \subset M$ be a knot or link in $M$. Let $M^{(1)}$ be a
regular neighborhood of $K$ and declare each component to be a
(floating) $1$-handle. Let $\Sigma$ be a Seifert surface for $K$, and 
let $\{\Sigma_j\}$ be the components of $\Sigma \cap (M -
\interior(M^{(1)}))$. The $2$-handles will be regular neighborhoods of
the surfaces $\Sigma_j$ in $M - \interior(M^{(1)})$. Lastly, the
$3$-handles will be the components of $M - \interior(M^{(2)})$. This 
decomposition has no $0$-handles, and its $1$-handles are (obviously)
knotted. 
\end{example}

Now suppose $M$ is an $n$-manifold with boundary. A \emph{generalized
handle decomposition of $M$} is a pair of filtrations $\emptyset =
M^{(-1)} \subset M^{(0)} \subset \cdots \subset M^{(n)} = M$ and
$\emptyset = N^{(-1)} \subset N^{(0)} \subset \cdots \subset N^{(n-1)} =
\partial M$ by codimension-zero submanifolds, such that: 
\begin{itemize}
\item[(i)] the filtration $\emptyset = N^{(-1)} \subset N^{(0)} \subset
\cdots \subset N^{(n-1)} = \partial M$ is a generalized handle
decomposition of $\partial M$, 
\item [(ii)] for each $i$, $M^{(i)}$ is obtained from $M^{(i-1)} \cup
N^{(i-1)}$ by attaching finitely many index $i$ handles, and  
\item[(iii)] each index $i-1$ handle of $\partial M$ is a connected
component of the intersection of $\partial M$ with an index $i$ handle
of $M$. In particular, $N^{(i-1)} = \partial M \cap M^{(i)}$ for all $i$. 
\end{itemize}
In (ii), each handle $H = \Sigma^{i} \times D^{n-i}$ is attached via
an embedding $h \co (\partial \Sigma^{i} \times D^{n-i}) \to (\partial
M^{(i-1)} \cup N^{(i-1)})$. As before, we require the images of the
attaching maps of the index $i$ handles to be disjoint. It follows
that the individual $i$-handles are embedded in $M$, and are disjoint
from each other. 

Let $f \co M \to X$ be a map from a compact $n$-manifold to a CW
complex. We say that $f$ is \emph{transverse} to the cell structure of
$X$ if $M$ has a generalized handle decomposition such that the
restriction of $f$ to each handle is given by projection onto the second
factor, followed by the characteristic map of a cell of $X$. Thus, index
$i$ handles map to $(n-i)$-dimensional cells. In particular, $M$ maps
into the $n$-skeleton of $X$. In a transverse map there may be
floating handles of any index, and it may not be possible to modify $f$
to eliminate these. By the same token, one must always allow for the
possibility of knotted handles. 

One virtue of transverse maps is that they can easily be proved to
exist. However, to accomplish this, we must assume additional structure
on the target complex $X$. We say that $X$ is a \emph{transverse CW
complex} if the attaching map of every cell is transverse to the cell
structure of the skeleton to which it is attached. The main existence
result is the following: 

\begin{transthm}[Buoncristiano-Rourke-Sanderson]
Let $M$ be a compact smooth manifold and $f\co M \to X$ a continuous map
into a transverse CW complex. Suppose $f \vert_{\partial M}$ is
transverse. Then $f$ is homotopic rel $\partial M$ to a transverse map $g
\co M \to X$. 
\end{transthm}

The theorem includes the case where $M$ is closed: all maps of closed
manifolds can be made transverse by a homotopy. 

This theorem is proved in \cite{BuRoSa} for PL manifolds, and the proof
in the smooth case is entirely analogous. The proof is a step by step
application of smooth transversality, applied to preimages of open cells
(considered as smooth manifolds themselves), starting with the top
dimensional cells and working down. The first stage of the argument, in
which the $0$-handles are constructed, is explained fully in the proof of
Lemma~2.3 of \cite{bbfs}. This is precisely the construction of
admissible maps (defined below). 

\begin{remark}
In order to apply the theorem one needs transverse CW complexes. Any CW
complex can be made transverse by successively homotoping the attaching
maps of its cells (by the Transversality Theorem and induction on
dimension); this procedure 
preserves homotopy type. Moreover, in this paper, the complex $X$ that we
use can 
be made transverse in a more direct and controlled way, preserving both
its homeomorphism type and its partition into open cells; see
Section~\ref{3cx} and Figure~\ref{fig:cell}. 
\end{remark}

\subsection*{Admissible maps and combinatorial volume} 
Recall from \cite{bbfs} the definition of an \emph{admissible map}: it
is a map $f \co M^n \to  X^{(n)} \subset X$ such that the preimage of
every open $n$-cell is a disjoint union of open $n$-dimensional balls
in $M$, each mapped by $f$ homeomorphically onto the $n$-cell. The
\emph{combinatorial volume} of an admissible map, denoted $\Vol^n(f)$, is
the number of open balls mapping to $n$-cells. 

It is clear that transverse maps are admissible: the interiors of
$0$-handles are open balls, and the rest of $M$ maps into
$X^{(n-1)}$. 
Conversely, if one applies the proof of the transversality 
theorem to an admissible map to make it transverse, then the preimages of
the $n$-cells will not change (except possibly by being shrunk
slightly), and combinatorial volume is preserved. For this reason, given
an admissible map, the closures of the open balls mapping to $n$-cells
will be called $0$-handles. 

Note that in an admissible map, $0$-handles may intersect each other in
their boundaries. For example, if $M$ has a cell structure, then the
identity map is admissible, with $0$-handles equal to the closures of the
top-dimensional cells. 

In \cite[Lemma 2.3]{bbfs} it is shown that every map from a smooth or PL
manifold is homotopic to an admissible map. This is a special case of the
Transversality Theorem, though it is not required that the target CW
complex be transverse. The existence of admissible maps can also be
proved without relying on a smooth or PL structure; see Epstein
\cite[Theorem 4.3]{epstein}. 

\subsection*{Volume reduction} 
In this paper, generalized handle decompositions (and transverse maps)
will serve as  higher-dimensional analogues of van Kampen
diagrams. Indeed, in dimension $2$, transverse maps already provide an
alternative to the combinatorial approach to diagrams, and they have
several advantages. This is the viewpoint taken in \cite{cpr} and
\cite{stallings}, for example. 
With van Kampen diagrams one often considers \emph{reduced} diagrams,
where no folded cell pairs occur. The same type of cancellation process
also works for admissible and transverse maps. One such process is given
as follows. 

Let $f \co M^n \to X$ be an admissible map, and let $H_0,H_1 \subset M$
be $0$-handles, and $\alpha \subset M - (\interior(H_0) \cup
\interior(H_1))$ a $1$-dimensional submanifold homeomorphic to an
interval, with endpoints in $H_0$ and $H_1$ (we also allow the
degenerate case in which $\alpha$ is a point in $H_0 \cap H_1$). Suppose
that $f$ maps $\alpha$ to a point and maps $H_0$ and $H_1$ to the same
$n$-cell, with opposite orientations (relative to a neighborhood of $H_0
\cup \alpha \cup H_1$, which is always orientable). Since $H_0$ and
$H_1$ are $0$-handles, there are homeomorphisms $h_i \co H_i \to D^n$
such that $f \vert_{H_i} = \Phi \circ h_i$ for some characteristic map
$\Phi\co D^n \to X$.  Now delete interiors of $H_i$ from $M$ and then
form a quotient $M'$ by gluing boundaries via $h_0^{-1} \circ h_1$ and
collapsing $\alpha$ to a point. The new space maps to $X$ by $f$, and
there is a homeomorphism $g \co M\to M'$. Now $f \circ g$ is an
admissible map $M \to X$ with two fewer $0$-handles. Note that the other
$0$-handles are unchanged. If desired, this new map can then be made
transverse, with the same $0$-handles, and with its (lowered) volume
unchanged. 

\begin{remark}
There is, in fact, a more general procedure for cancelling $H_0$ and $H_1$
that does not require $\alpha$ to map to a point. This procedure is due
to Hopf \cite{hopf2} and a detailed treatment was given by Epstein 
\cite{epstein}. If $X$ is $2$-dimensional then the more general procedure
is not particularly useful: new $0$-handle pairs can be created when
cancelling $H_0$ and $H_1$, and volume may fail to decrease. In higher
dimensions, however, no new $0$-handle pairs are created and the volume
will always decrease by $2$. 
\end{remark}

\subsection*{Riemannian volume} 

If $N$ is a smooth manifold, $M$ an oriented Riemannian manifold of the same
dimension, and $f\co N \to M$ a smooth map, then the \emph{volume} of $f$
can be defined. Following Gromov \cite[Remarks 2.7 and 2.8{\small
$\frac{1}{2}$}]{gromov:metric}, let $\nu_M$ be the volume form on $M$ and
choose any Riemannian metric on $N$. We define
\[ \RVol(f) \ = \ \int_N f^* (\abs{\nu_M}).\]
The integral is independent of the choice of metric on $N$, by the change
of variables formula. Note that we are using $\abs{{\rm vol}}(f)$, not ${\rm
vol}(f)$, in the notation of \cite{gromov:metric}. (The latter allows
cancellation of volume, which is not appropriate in our setting.) In
fact, we need not assume that $M$ is oriented, since $\abs{\nu_M}$ is
still defined. If $\dim N = 2$ then $\RVol$ is also denoted $\RArea$. 

If $f$ is an immersion then this definition amounts to giving $N$ the
pullback metric and taking the volume of $N$. More generally, if $f$
fails to be an immersion at some $x \in N$, then $f^*(\abs{\nu_M})$ is
zero at $x$, and does not contribute to volume. Hence, $\RVol(f)$ is the
volume of the pullback metric on $U \subset N$, the set on which $f$ is
an immersion. Note that $U$ is open, and hence is a Riemannian
manifold. Generically, $U$ has full measure in $N$ when $\dim N \leq \dim M$
\cite[1.3.1]{gromov:pdr}. 

From this perspective, we can now define $\RVol(f)$ when $\dim N \not=
\dim M$. We define $\RVol(f)$ to be the volume of $U\subset N$, 
the set on which $f$ is an immersion, with the pullback metric. Note that
$\RVol(f)$ measures \emph{$n$-dimensional volume}, where $n = \dim N$. 

Lastly, we wish to extend the definition of volume to allow a piecewise
Riemannian CW complex in place of $M$. The complex $\widetilde{X}$ that
interests us is a $3$-complex with branching locus a $2$-manifold,
homeomorphic to the product of $\R^2$ with a simplicial tree. In a
neighborhood of any singular point one sees a union of half-spaces joined
along their boundaries, naturally grouped into two collections, with a
well defined common tangent space at the singular point. The situation is
similar to that of a train track, or a branched surface from lamination
theory (eg. \cite[Section 6.3]{calegari}). There is a smooth  structure, and
$\widetilde{X}$ comes equipped with an immersion $q \co\widetilde{X} \to
M$ onto a Riemannian manifold $M$. (This immersion is not locally
injective, but is injective on tangent spaces.) The Riemannian 
metric on $\widetilde{X}$ is the pullback under $q$ of the metric on
$M$. The volume $\RVol(f)$ can now be defined directly (as above) using
this metric on $\widetilde{X}$, or equivalently by defining $\RVol(f) =
\RVol(q \circ f)$. 

\begin{remarks}
(1) If $\dim N > \dim M$ (or $\dim N > \dim \widetilde{X}$) then
    $\RVol(f)$ is zero, since $f$ is an immerison nowhere. Similarly, if
    $f$ factors through a manifold of smaller dimension, then the volume
    is zero. 

(2) Any transverse map $f \co N \to \widetilde{X}$ is piecewise smooth,
    and is a submersion on each handle. It will be an immersion
    only on the $0$-handles. This latter statement also holds for
    admissible maps, since the complement of the $0$-handles is mapped into a
    lower-dimensional skeleton. 
\end{remarks}

\begin{remark}\label{minRvol}
We will be interested in finding least-volume maps extending a given
boundary map. If the set of volumes of $n$-cells of a piecewise Riemannian
CW complex is finite, then least-volume transverse maps of $n$-manifolds exist
in any homotopy class. This is because the Riemannian volume of a
transverse map is a \emph{positive} linear combination of numbers in this
set, and hence the set of such volumes is discrete, and well-ordered. 
\end{remark}

\subsection*{Dehn functions} 
Here we recall the definition of the $n$-dimensional Dehn function of a
group from \cite{bbfs}. Note that these definitions all use combinatorial
volume. Given a group $G$ of type $\mathcal{F}_{n+1}$, fix an aspherical
CW complex $X$ with fundamental group $G$ and finite
$(n+1)$-skeleton (the existence of such an $X$ is the meaning of
``type $\mathcal{F}_{n+1}$''). Let $\widetilde{X}$ be the universal cover
of $X$. If 
$f\co S^n \to \widetilde{X}$ is an admissible map, define the
\emph{filling volume} of $f$ to be the minimal volume of an admissible
extension of $f$ to $B^{n+1}$: 
\[ \, \FVol(f) \ = \ \min \{\, \Vol^{n+1}(g) \mid g \co B^{n+1} \to 
\widetilde{X}, \  g\vert_{\partial B^{n+1}} = f \, \}.\]  
Note that extensions exist since $\pi_n(\widetilde{X})$ is trivial, and
any extension can be made admissible, by \cite[Lemma 2.3]{bbfs}. We
define the  \emph{$n$-dimensional Dehn function} of $X$ to be 
\[\delta^{(n)}(x) \ = \ \sup \{ \, \FVol(f) \mid f \co S^n \to
\widetilde{X}, \ \Vol^n(f) \leq x \, \}.\]
Again, the maps $f$ are assumed to be admissible. 

In \cite{AWP} it was shown that $\delta^{(n)}(x)$ is finite for each $x\in
\N$, and that, up to coarse Lipschitz equivalence, $\delta^{(n)}(x)$ depends
only on $G$. Thus the Dehn function will sometimes be denoted
$\delta^{(n)}_G(x)$. (Recall that functions $f,g \co \R_+ \to \R_+$ are
\emph{coarse Lipschitz equivalent} if $f \preceq g$ and $g \preceq f$,
where $f \preceq g$ means that there is a positive constant  $C$ such
that $f(x) \leq C\,g(Cx) + Cx$ for all $x\geq 0$.) If we wish to specify
$\delta^{(n)}(x)$ exactly, we may denote it as $\delta^{(n)}_X(x)$. 

Taking $n=1$ yields the usual Dehn function $\delta(x)$ of a group $G$.

\subsection*{The strong Dehn function}
The notion of $n$-dimensional Dehn function was modified in \cite{bbfs}
to allow fillings by compact manifolds other than the ball $B^{n+1}$. In
this way, every compact manifold pair $(M, \partial M)$ gave rise to a
Dehn function $\delta^M(x)$. Several of the main results proved in
\cite{bbfs} had hypotheses and conclusions involving the functions
$\delta^M(x)$ ``for all $n$-manifolds $M$.'' An equivalent way of
formulating these results is by means of the \emph{strong Dehn function},
defined as follows. 

Given a compact $(n+1)$-manifold $M$ and an admissible map $f \co
\partial M \to \widetilde{X}$, define 
\[\FVol^M  (f) \ = \ \min \{\, \Vol^{n+1} (g) \mid g \co M \to
\widetilde{X} \text{ admissible}, \ g \vert_{\partial M} = f\, \}\]
and 
\begin{align*}
\Delta^{(n)}(x) \ = \ \sup \{ \, \FVol^M(f) \mid \ &(M,\partial M) \text{
is a compact $(n+1)$-manifold}, \\
&\quad \quad f\co \partial M \to \widetilde{X} \text{ admissible}, \
\Vol^n(f) \leq x \, \}.
\end{align*} 
We call $\Delta^{(n)}(x)$ the \emph{strong $n$-dimensional Dehn function}
of $X$. Note that the manifolds $M$ appearing in the definition are not 
assumed to be connected. The statement $\Delta^{(n)}(x) \leq y$ means 
that for every compact manifold $(M, \partial M)$ and every admissible
map $f \co \partial M \to \widetilde{X}$ of volume at most $x$, there is
an admissible extension to $M$ of volume at most $y$. In particular, the
bound $y$ is uniform for all topological types of fillings (hence the
word ``strong''). Note that this is very different from
\emph{homological} Dehn functions, where only a single filling by an
$(n+1)$-cycle is needed, of some topological type. 

The strong Dehn function has two principal features. The first is that it
behaves well with respect to splittings and mapping torus constructions
(as does the homological Dehn function). The next two theorems below are
examples of this phenomenon. The second is that it (clearly) satisfies 
\begin{equation}\label{upperstrong}
\delta^{(n)}(x) \ \leq \ \Delta^{(n)}(x)
\end{equation}
and hence it may be used to establish upper bounds for
$\delta^{(n)}(x)$. To this end, the following two theorems are proved in
\cite{bbfs} (Theorems 7.2 and 8.1). 

\begin{theorem}[Stability for Upper Bounds]\label{suspension}
Let $X$ be a finite aspherical CW complex of dimension at most $n+1$. Let
$f \co X \to X$ be a $\pi_1$-injective 
map and let $Y$ be the mapping
torus of $X$ using $f$. Then $\Delta^{(n+1)}_Y(x) \leq
\Delta^{(n)}_X(x)$. 
\end{theorem}
Thus, any upper bound for $\Delta^{(n)}_X(x)$ remains an upper bound for
$\Delta^{(n+1)}_Y(x)$. 
A similar result holds more generally (with the same proof) if $Y$ is the
total space of a graph of spaces whose vertex and edge spaces satisfy the
hypotheses of $X$. Then the conclusion is that $\Delta^{(n+1)}_Y(x) \leq
C \, \Delta^{(n)}_X(x)$ for some $C > 0$. 

The next result provides a better bound in a special case. 

\begin{theorem}[Products with $S^1$] \label{product}
Let $X$ be a finite aspherical CW complex of dimension at most $n+1$. If
$\Delta^{(n)}_X(x) \leq Cx^s$ for some $C>0$ and $s \geq 1$ then
$\Delta^{(n+1)}_{X \times S^1}(x) \leq C^{1/s} x^{2-1/s}$. 
\end{theorem}

It turns out that for $n \geq 3$ and for $n=1$, there is no significant
difference between the strong and ordinary Dehn functions. The precise
relation between them is stated in the next theorem, which was
essentially proved already in Remark 2.5(4) and Lemma 7.4 of
\cite{bbfs}. 

However, we do indeed need to work specifically with the strong Dehn
function in dimension $2$, since we wish to apply Theorem \ref{product}
above. This case forms the base of the induction argument we use to show
that $\IP^{(n)}$ is dense for all $n \geq 2$. 

A function $f \co \N \to \N$ is \emph{superadditive} if $f(a) + f(b) \leq
f(a+b)$ for all $a,b \in \N$. The \emph{superadditive closure} of $f$ is
the smallest superadditive $g$ such that $f(x) \leq g(x)$ for all $x$. 
An explicit recursive definition of $g$ is given by 
\[ g(0) \ = \ f(0), \quad g(x) \ = \ \max \bigl\{ \{g(i) + g(x-i) \mid i
= 1, \ldots, x-1\} \cup \{ g(0) + f(x)\} \bigr\}.\]
It is easy to verify that $\Delta^{(n)}(x)$ is always superadditive, by
considering fillings by non-connected manifolds. 

\begin{theorem}
$\Delta^{(n)}_X(x)$ is the superadditive closure of $\delta^{(n)}_X(x)$
for $n \geq 3$ and for $n=1$.  
\end{theorem}

It is not known whether there exist groups $G$ for which
$\delta^{(n)}_G(x)$ is not superadditive (up to coarse Lipschitz
equivalence). Indeed, when $n=1$, Sapir has conjectured that this does
not occur \cite{gs}. So in all known examples,
$\Delta^{(n)}$ and $\delta^{(n)}$ agree (for $n \geq 3$ or $n=1$). 

In contrast, 
Young \cite{young} has shown that the statement of the theorem is
false when $n=2$. Specifically, he shows that for a certain group $G$,
the strong Dehn function $\Delta^{(2)}_G(x)$ is not bounded by a
recursive function, whereas $\delta^{(2)}_G(x)$ always satisfies such a
bound, by Papasoglu \cite{papa}. The superadditive closure will inherit
this property, since it is computable from $\delta^{(2)}_G(x)$. 

\begin{proof}
Let $s(x)$ be the superadditive closure of $\delta^{(n)}(x)$. 

If $n=1$ then the proof of Lemma 7.4 of \cite{bbfs} shows directly that
for any compact $2$-manifold $M$, one has $\delta^M(x) \leq \delta^{D^2 \sqcup
\cdots \sqcup D^2}(x)$, where the number of disks equals the number of
boundary components of $M$. For each admissible $f \co S^1 \sqcup \cdots
\sqcup S^1 \to X$ with length $x = \sum_i x_i$ we have $\FVol^{D^2 \sqcup
\cdots \sqcup D^2}(f) \leq \sum_i \delta^{(1)}(x_i) \leq s(x)$, and so
$\delta^M(x) \leq s(x)$. Therefore $\Delta^{(1)}(x) \leq s(x)$. Since
$\Delta^{(1)}(x)$ is superadditive and $\delta^{(1)}(x) \leq 
\Delta^{(1)}(x)$, it follows that $\Delta^{(1)}(x) = s(x)$. 

If $n \geq 3$ then the argument given in Remark 2.5(4) of \cite{bbfs}
applies. Let $\{N_i\}$ be the components of $\partial M$ and suppose that
$g_i \co N_i \to X$ are admissible maps of volume $x_i$, with union $g
\co \partial M \to X$ of volume $x = \sum_i x_i$. By the argument given
in \cite{bbfs}, for each $i$ there is an admissible homotopy of
$(n+1)$-dimensional volume at most $\delta^{(n)}(x_i)$ to an admissible
map $g_i' \co N_i \to X$ with image inside $X^{(n-1)}$. The union of
these maps can be filled by a map $M \to X^{(n)}$, since $X^{(n-1)}$ is
contractible inside $X^{(n)}$. This filling has zero $(n+1)$-dimensional
volume, and hence $\FVol^M(g) \leq \sum_i \delta^{(n)}(x_i) \leq
s(x)$. Since $M$ and $g$ were arbitrary, we have $\Delta^{(n)}(x) \leq
s(x)$, and hence $\Delta^{(n)}(x) = s(x)$. 
\end{proof}

\begin{remark}[Lower bounds] \label{embedded}
As noted earlier, the strong Dehn function can be used to bound
$\delta^{(n)}(x)$ from 
above. For a lower bound one needs explicit information about
$\FVol(f)$ for admissible maps $f \co S^n \to \widetilde{X}$. That is,
one needs to identify \emph{least-volume} extensions $g \co B^{n+1} \to
\widetilde{X}$. Suppose $\dim \widetilde{X} = n+1$ and
$H_{n+1}(\widetilde{X};\Z) =0$. Then a simple homological argument,
sketched in Remarks 2.2 and 2.6 of \cite{bbfs}, shows that $g$ is
least-volume if $g$ is injective on the interiors of $0$-handles (i.e. no
two $0$-handles map to the same cell of $\widetilde{X}$). 
For convenience we provide the full argument here. 

Let $C_{n+1}(\widetilde{X})$ be the cellular chain group for
$\widetilde{X}$. Given an oriented manifold $M^{n+1}$ and a transverse
map $f \co M^{n+1} \to \widetilde{X}$, there is a chain
$[f] \in C_{n+1}(\widetilde{X})$ defined 
as follows. For each $(n+1)$-cell $e_{\alpha}$, let $\sigma_{\alpha}$ be
the corresponding generator of $C_{n+1}(\widetilde{X})$ 
and define $d_{\alpha}(f)$ to be the local degree of $f$ at
$e_{\alpha}$ (i.e. the number of $0$-handles of $f$ mapping to
$e_{\alpha}$, counted with respect to orientations). We define 
$[f] = \sum_{\alpha} d_{\alpha}(f) \sigma_{\alpha}$. Note that the
boundary of $[f]$ in $C_n(\widetilde{X})$ is simply $[f \vert_{\partial
M}]$. (Here the transversality structure is used: $0$-handles in
$\partial M$ are joined to $0$-handles in $M$ by $1$-handles, compatibly
with boundaries of characteristic maps of cells in $\widetilde{X}$.) 

Now suppose that $g \co B^{n+1} \to \widetilde{X}$ is injective on
$0$-handles, and $h \co B^{n+1} \to \widetilde{X}$ is another transverse
map with $h\vert_{S^n} = g \vert_{S^n}$. These maps together define a
transverse map $g -h \co S^{n+1} \to \widetilde{X}$ by considering
$S^{n+1}$ as a union of two balls, with the orientation on one of the
balls reversed. We have $[ g -h] = [g] - [h]$ in
$C_{n+1}(\widetilde{X})$, and so $\partial [g -h] = \partial[g]
- \partial [h] = 0$, and $[g -h]$ is a cycle. Since
$H_{n+1}(\widetilde{X}) = 0$ and $C_{n+2}(\widetilde{X}) = 0$, 
this cycle must be zero in $C_{n+1}(\widetilde{X})$. That is, $g -h$ has
zero local degree at every $(n+1)$-cell. Hence ${d_{\alpha}(g)} =
{d_{\alpha}(h)}$ for all $\alpha$. 

The injectivity assumption on $g$ implies that $\Vol^{n+1}(g) =
\sum_{\alpha} \abs{d_{\alpha}(g)}$. Then we have
\[ \Vol^{n+1}(h) \ \geq \ \sum_{\alpha} \abs{d_{\alpha}(h)} 
\ = \ \sum_{\alpha} \abs{d_{\alpha}(g)} \ = \ \Vol^{n+1}(g),\] 
and hence $g$ is least-volume. 
\end{remark}

\section{The groups $G_A$ and their model spaces} 

\subsection*{The model manifold $M$} 
Let $M$ be the manifold $\R^3$ with the metric $ds^2 = \la^{-2z}dx^2 +
\mu^{-2z}dy^2 + dz^2$, where $\la > 1$, $\mu < 1$, and $\la \mu >1$. This
is the left-invariant metric for the solvable Lie group $\R^2 \rtimes
\R$, with $z\in \R$ acting on $\R^2$ by the matrix $\big( \,
{}^{\lambda^z}_0 \ {}^0_{\mu^z}  \big)$. The geometry of $M$ has much in
common with that of \textsc{Sol} (the case $\lambda\mu = 1$), but with
some important differences. 



\subsection*{The group $G_A$ and its model space $X$}\label{3cx}
Let $A \in M_2(\Z)$ be a hyperbolic matrix with eigenvalues $\lambda >
1$ and $\mu <  1$ and determinant $d = \lambda \mu > 1$. Let $B\in
GL_2(\R)$ diagonalize $A$, so that $B A B^{-1} = \big( \, {}^{\lambda}_0 \
{}^0_{\mu} \, \big)$. Call this diagonal matrix $D$. Then $D$ 
preserves the lattice $\Gamma \subset \R^2$, defined to be the
image of $\Z \times \Z$ under $B$. 

Let $G_A$ be the ascending HNN extension of $\Z \times \Z$ with
monodromy $A$. That is, 
\[ G_A \ = \ \langle \, \Z\times \Z, t \mid t v t^{-1} = Av \ \text{ for
all } v \in \Z \times \Z \, \rangle.\] 
The matrix $B$ defines an isomorphism from $G_A$
to the (non-discrete) subgroup of $\R^2 \rtimes \R$ generated by $\Gamma$
and $1 \in \R$ (corresponding to the stable letter $t \in G_A$). 

The groups $G_A$ are the main examples that interest us in this paper;
our chief task will be determining their $2$-dimensional Dehn functions
$\delta^{(2)}(x)$. For this we need to construct a geometric model for
$G_A$. Note that $\R^2 \rtimes \R$ cannot serve as a model since the
subgroup $G_A$ is not discrete. (Indeed, this Lie group is not
quasi-isometric to \emph{any} finitely generated group, by 
\cite{EFW}.) 

Topologically, our model is formed from $T^2 \times I$ by glueing $T^2
\times 0$ to $T^2 \times 1$ by the $d$-fold covering map $T_A \co T^2 \to
T^2$ induced by $A$. To put a piecewise Riemannian metric on this space,
we use the geometry of $M$ as follows. The construction is analogous to
building the standard presentation $2$-complex of a Baumslag-Solitar
group from a ``horobrick'' in the hyperbolic plane \cite{fm}. 

Let $Q \subset \R^2$ be the parallelogram spanned by the generators of
$\Gamma$. Then $Q \times [0,1]$ is a fundamental domain for the action of
$\Gamma$ on $\R^2 \times [0,1] \subset \R^2 \rtimes \R$, with quotient
homeomorphic to $T^2 \times [0,1]$. The isometry $\R^2 \times 0 \to \R^2
\times 1$ given by $(x,y,0) \mapsto (\lambda x, \mu y, 1)$ is
$\Gamma$-equivariant and induces a local isometry $\R^2/ \Gamma
\times 0 \to \R^2/ \Gamma \times 1$. This local isometry agrees
precisely with the map $T_A \co T^2 \to T^2$ under the identification of
$\R^2 / \Gamma$ with $T^2$ induced by $B$. Thus, identifying opposite
sides of $Q \times [0,1]$ to obtain a copy of $T^2 \times [0,1]$, the
glueing $T^2 \times 0 \to T^2 \times 1$ is locally isometric, and the
model for $G_A$ is a piecewise Riemannian space. Call it $X$, and its
universal cover $\widetilde{X}$. 

Figure~\ref{fig:Q} below shows $Q$ and the locally isometric glueing map
for the example $A = \big( \, {}^{4}_1 \ {}^2_{1} \, \big)$. 
The diagonal matrix stretches horizontally and compresses vertically. 
\begin{figure}[ht]
\psfrag{Qx0}{{$Q \times 0$}}
\psfrag{Qx1}{{$Q \times 1$}}
\psfrag{M}{\hspace{-.17in}$\big( \, {}^{\lambda}_0 \  {}^0_{\mu} \, \big)$}
\includegraphics{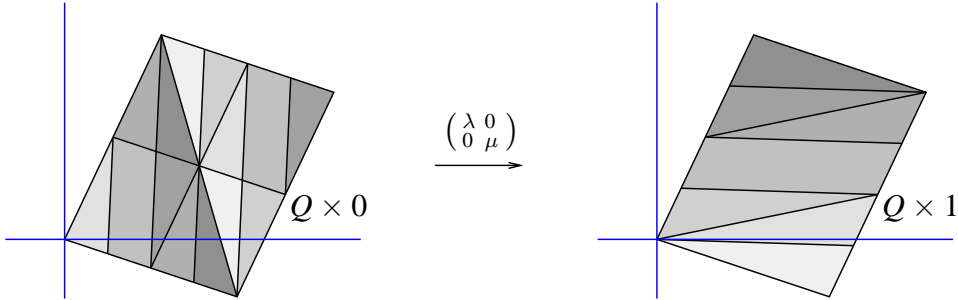}
\caption{The region $Q$ and the glueing map given by the diagonalized
form of $A = \big( \, {}^{4}_1 \ {}^2_{1} \, \big)$. 
Also shown is a cell structure (discussed below) for which this map is
combinatorial.} 
\label{fig:Q}
\end{figure}

\begin{noname}\label{covergeometry}
The cover $\widetilde{X}$ is tiled by isometric copies of $Q \times
[0,1]$, with tiles meeting isometrically along faces. A generic 
point in the top face $Q \times 1$ of a tile meets $d$ tiles in their
bottom faces; side faces are joined in pairs. Topologically, 
$\widetilde{X}$ is a branched space homeomorphic to $\R^2 \times T$,
where $T$ is the Bass-Serre tree corresponding to the splitting of $G_A$
as an ascending HNN extension. The $G_A$-tree $T$ has a fixed end $\eta$
and there is an equivariant map $h_0\co T \to \R$, sending $\eta$ to
$-\infty$ and all other ends to $\infty$, such that the induced
$G_A$-action on $\R$ is by integer translations. The preimage of $\Z$
under this map is the set of vertices of $T$. 

There is a locally isometric surjection $q\co \widetilde{X} \to M$ which,
viewed via the homeomorphisms $\widetilde{X} \cong \R^2 \times T$ and $M \cong
\R^2 \times \R$, is given by the identity on $\R^2$ and the map $h_0\co T\to
\R$ described above. The metric on $\widetilde{X}$ may be viewed as the
pullback metric of $M$ under this map. In particular, for any compact
manifold $W$ and any piecewise smooth map $f \co W \to \widetilde{X}$, we
have $\RVol(f) = \RVol(q \circ f)$. 

If $L \subset T$ is a line mapping homeomorphically to $\R$ under 
$h_0$, then the subspace $\R^2 \times L \subset \widetilde{X}$ is 
isometric to $M$. This situation is completely analogous to that of 
the solvable Baumslag-Solitar groups, whose standard geometric models
contain copies of the hyperbolic plane (cf. \cite{fm}). 

The map $h_0 \co T \to \R$ also defines a \emph{height function} $h \co
\widetilde{X} \to \R$ by composing with the projection $\widetilde{X}
\cong \R^2 \times T \to T$. 
\end{noname}

\subsection*{Cell structure} 
The basic cell structure on $X$ is the usual mapping torus cell
structure, induced by the standard cell decomposition for the torus, but
we will need to modify the attaching maps to make it a transverse CW
complex. 

First, consider $Q \times [0,1]$ combinatorially as a cube and give it
the product cell structure (with eight $0$-cells, twelve $1$-cells, six
$2$-cells, and one $3$-cell). The side-pairings are compatible with this
structure, so we have a cell structure on $T^2 \times [0,1]$. Now
subdivide the top and bottom faces $T^2 \times \{0,1\}$ into finitely
many cells so that $T_A \co T^2 \times 0 \to T^2 \times 1$ maps open
cells homeomorphically to open cells (i.e. $T_A$ becomes a
\emph{combinatorial map}). Note that $T^2 \times 0$ will have $d$ times 
as many $2$-cells as $T^2 \times 1$, since $T_A$ is a $d$-fold covering. 
The pattern of subdivision is obtained by taking intersections of cells
of $T^2 \times 1$ with cells of $T_A(T^2 \times 0)$. See
Figure~\ref{fig:Q} for the example $A = \big( \, {}^{4}_1 \ {}^2_{1} \,
\big)$. Since $T_A$ takes cells to cells, we now have a cell structure on
$X$. 

Next we make the cell structure transverse. In this case, the
transversality procedure does not change the homeomorphism type of $X$,
or even its partition into open cells. Thus, the piecewise Riemannian
metric will still exist, exactly as described, 
with either cell structure. 

Every map $S^0 \to X^{(0)}$ is transverse, so the $1$-skeleton $X^{(1)}$
is already a transverse CW complex. For the $2$-skeleton, note that for
each attaching map $S^1 \to X^{(1)}$ in the original cell structure,
there is a realization of $S^1$ as a graph such that the map is a graph
morphism. To make this map transverse, expand each vertex into a closed
interval (a $1$-handle) to form a slightly larger circle. Let the new
attaching map first collapse these intervals back into vertices, and then
map to $X^{(1)}$ by the original attaching map. We have simply introduced
some ``slack'' at the vertices. The $2$-skeleton and its partition into
open cells has not changed. 

For the attaching map $S^2 \to X^{(2)}$ of the $3$-cell, note again that
$S^2$ has a cell structure for which this map is combinatorial (this is a
property of our particular complex $X$). Expand every $0$-cell into a
small disk (a $2$-handle) and then expand every $1$-cell into a rectangle
(a $1$-handle), to abtain a new copy of $S^2$. The new transverse
attaching map will collapse these new handles to $0$- and $1$-cells and
then map to $X^{(2)}$ as before. See Figure~\ref{fig:cell}.
Again, the 
topology of $X$ is unchanged. (This amounts to a claim that performing
the collapses described above in the boundary of a ball results again in
a ball.) 
\begin{figure}[ht]
\psfrag{ar}{{$\longrightarrow$}}
\includegraphics{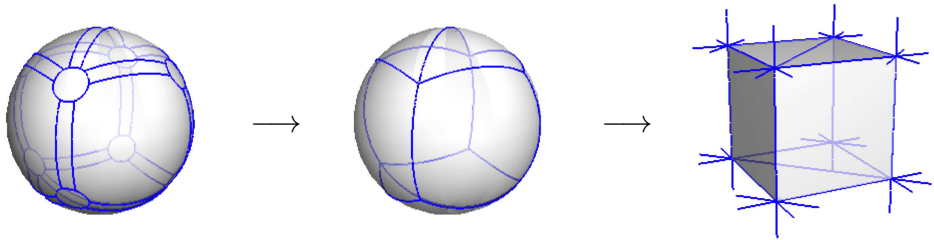}
\caption{Transverse $3$-cell attachment. The rightmost map is the
original attaching map; the composition is the new (transverse) one.} 
\label{fig:cell}
\end{figure}
%
%
%

The universal cover $\widetilde{X}$ is given the induced cell
structure. Note that the closures of the $3$-cells are exactly the copies
of $Q \times [0,1]$ tiling $\widetilde{X}$ mentioned earlier. Also note
that every $2$-cell is either \emph{horizontal} or \emph{vertical}: in
the product $\R^2 \times T$, it either projects to a point in $T$ or to a
line segment in $\R^2$. In the latter case, the projection of the
$2$-cell in $T$ is exactly an edge.

\section{The upper bound}\label{uppersec}

We proceed now to establish an upper bound for the strong Dehn
function $\Delta^{(2)}(x)$ of the group $G_A$. 

Let $W$ be a compact $3$-manifold with boundary and $f \co \partial W
\to \widetilde{X}$ an admissible map, which we may make transverse
without changing its combinatorial area (by a homotopy inside
$\widetilde{X}^{(2)}$, of zero volume). Now let $g\co W \to
\widetilde{X}$ be a transverse extension of $f$ of smallest Riemannian
volume (cf. Remark~\ref{minRvol}). 

We need to measure the combinatorial volume of $g$ and bound it in terms
of the area of $f$. Note that every $0$-handle of $W$ has the same
Riemannian volume, equal to the volume $V$ of the single $3$-cell in
$X$. Thus, to count the $0$-handles, we will instead measure the 
Riemannian volume of $g$ by integration and divide by $V$. It turns out
that the geometry of $\widetilde{X}$ is well-suited to this kind of
measurement. We will also work with the Riemannian area of $f$, but
again the relation to combinatorial area causes no difficulty.

\subsection*{The embedded case} First we discuss a special case in order
to clarify the geometric ideas, before incorporating transverse maps into
the argument. We will assume that $W$ is a subcomplex of $\widetilde{X}$,
with $g$ the inclusion map. 

Since $W$ is a manifold, every $2$-cell of $W$ is either in $\partial W$
or is adjacent to two $3$-cells of $W$. Let $F \subset W$ (the \emph{fold
set}) be the smallest subcomplex whose $2$-cells are the horizontal
$2$-cells $\sigma$ such that $\sigma \not\subset \partial W$ and both
adjacent $3$-cells are \emph{above} $\sigma$ with respect to the height
function $h\co \widetilde{X} \to \R$. (The fold set may be empty, of
course.) 

\begin{proposition}\label{foldbound}
$\RVol(W) \ \leq \ \frac{1}{\ln(\lambda \mu)}(\area(\partial W) + 2
\area(F))$. 
\end{proposition}

\begin{proof}
In $M$, integrating the volume element $(\lambda \mu)^{-z} \, dx dy
dz$ along a vertical ray from $z=0$ to $z=\infty$ yields
$\frac{1}{\ln(\lambda \mu)}$ times $dxdy$, the horizontal area element at
the initial point of the ray. Also, at any point of $\partial W$, the
surface area element is greater than or equal to the horizontal area
element. 

Consider a flow on $\widetilde{X} \cong \R^2 \times T$ which is towards
the end $\eta$ in the $T$ factor and the identity in $\R^2$. This flow is
semi-conjugate (by $q$) to a flow in $M$ which is directly
downward. Under this flow, every point $p$ of $W$ leaves $W$, either
through $\partial W$ or through $F$. Let $\pi_-(p)$ be the first point of
$\partial W$ or $F$ that $p$ meets under this flow. This defines a map
$\pi_- \co W \to (\partial W \cup F)$, not necessarily continuous. Then
$W$ decomposes into two parts, $W_{\partial} = \pi_-^{-1}(\partial W)$
and $W_F = \pi_-^{-1}(F)$. 

For any $p\in \partial W$, the fiber $\pi_-^{-1}(p)$ is a segment
extending upward from $p$, and integrating along these fibers, we find
that $\RVol(W_{\partial}) \leq \frac{1}{\ln (\lambda \mu)} \area(\partial
W)$. For $\RVol(W_F)$, the fiber of any point in $F$ 
consists of \emph{two} segments extending vertically, so $\RVol(W_F) \leq
\frac{2}{\ln( \lambda \mu)} \area(F)$. 
\end{proof}

It now suffices to bound $\area(F)$ from above in terms of
$\area(\partial W)$. 

\begin{noname}\label{items} 
We need to make some definitions. 
Let $L = \log_{\lambda}(\area(\partial W))$. We have the following
properties: 
\begin{align}
\lambda^L \ &= \ \area(\partial W), \label{lambdaL} \\
\mu^L \ &= \ \area(\partial W)^{\log_{\lambda}(\mu)}, \label{muL} \\
(\lambda \mu)^L \ &= \ \area(\partial W)^{1 +
\log_{\lambda}(\mu)}. \label{lambdamuL} 
\end{align}
Equation \eqref{lambdaL} holds by definition, \eqref{lambdamuL} follows
from \eqref{lambdaL} and \eqref{muL}, and \eqref{muL} is an instance of
the identity $a^{\log_b(c)} = c^{\log_b(a)}$. 

Let $v_1, \ldots, v_k \in V(T)$ be the vertices in the image of
$W$ under the projection $\pi_T \co \widetilde{X} \to T$. 
We define several items associated to these vertices: 
\begin{itemize}
\item $h_i = h_0(v_i)$, the \emph{height} of $v_i$ 
\item $F_i \ = \ \pi_T^{-1}(v_i) \cap F$, the \emph{fold set at $v_i$}  
\item $T_i = \{ x \in T \mid v_i \in [x,\eta) \, \}$, the \emph{subtree
above $v_i$} 
\end{itemize}
and the following subsets of $\partial W$:
\begin{itemize} 
\item $S_i = \partial W \cap \pi_T^{-1}(T_i)$, the \emph{surface above
$v_i$} 
\item $A_i = S_i \cap h^{-1}((h_i, h_i + 1))$, the \emph{low slice} of
$S_i$ 
\item $B_i = S_i \cap h^{-1}((h_i + L, h_i + L + 1))$, the \emph{high
slice} of $S_i$. 
\end{itemize}
Note that $\partial S_i$ has height $h_i$, so $A_i$ lies between heights
$0$ and $1$ above $\partial S_i$, and $B_i$ lies between heights $L$ and
$L+1$ above $\partial S_i$. 

\begin{lemma}\label{disjoint} 
$A_i \cap A_j = B_i \cap B_j = \emptyset$ for $i\not= j$. 
\end{lemma}

\begin{proof}
Consider the case of $A_i$ and $A_j$ first. If $h_i\not= h_j$
then $h(A_i) \cap h(A_j) = \emptyset$ since vertices have integer heights
and the sets $h(A_i)$ have the form $(h_i, h_i+1)$. If $h_i = h_j$
then $v_i \not\in T_j$ and $v_j \not\in T_i$, which implies that $T_i
\cap T_j = \emptyset$, and hence $A_i$ and $A_j$ are disjoint. The case
of $B_i$ and $B_j$ is similar. 
\end{proof}

Recall that for each $p \in F$, the fiber $\pi_-^{-1}(p)$ is a pair of
segments extending upward from $p$ (it is an open subtree of $p_0 \times
T \subset \R^2 \times T$, with no branching, since $W$
is a manifold). Define a (non-continuous) map $\pi_+ \co F \to \partial W$ 
by choosing $\pi_+(p)$ to be one of the two upper endpoints of the fiber
$\pi_-^{-1}(p)$ for each $p\in F$. Note that $\pi_+$ is injective (since
$\pi_- \circ \pi_+ = \id_F$), and 
$\pi_+(F_i) \subset S_i$. The choices of endpoints can be made so that
$\pi_+$ is measurable. 

We now express each fold set $F_i$ as a union of two parts, the
\emph{low} and \emph{high} parts, as follows: 
\begin{align*}
(F_i)_{\text{\sl low}} \ &= \ \{p \in F_i \mid h(\pi_+(p)) \leq h_i + L + 1\}, \\
(F_i)_{\text{\sl high}} \ &= \ \{p \in F_i \mid h(\pi_+(p)) \geq h_i + L + 1\}. 
\end{align*}
Also define $F_{\text{\sl low}} = \bigcup_i (F_i)_{\text{\sl low}}$ and
$F_{\text{\sl high}}= \bigcup_i (F_i)_{\text{\sl high}}$. Clearly, $F =
F_{\text{\sl low}} \cup F_{\text{\sl high}}$. 
\end{noname}

\begin{proposition}\label{lowpart} 
$\area(F_{\text{\sl low}}) \leq (\lambda \mu) \area(\partial W)^{2 +
\log_{\lambda}(\mu)}$.  
\end{proposition}

\begin{proof}
We compare the areas of $F_{\text{\sl low}}$ and its image under $\pi_+$,
which is a subset of $\partial W$. Since $\pi_+$ projects points of
$F_{\text{\sl low}}$ upward a
distance of at most $L+1$, the horizontal area element at $p \in
F_{\text{\sl low}}$ is at most 
$(\lambda \mu)^{L+1}$ times the horizontal area element at
$\pi_+(p)$. Recall also that this latter area element is no larger than
the surface area element of $\partial W$ at $\pi_+(p)$. Since $\pi_+$ is
injective, we now have $\area(F_{\text{\sl low}}) \leq (\lambda
\mu)^{L+1} \area( \pi_+(F_{\text{\sl low}}))$. The proposition follows, by
equation~\eqref{lambdamuL} and the fact that
$\area(\pi_+(F_{\text{\sl low}})) \leq \area(\partial W)$. 
\end{proof}

\begin{noname}\label{furtherdefs} 
We need to introduce some further terminology. Recall that the map $q\co
\widetilde{X} \to M$ is the identity on the $\R^2$ factors of
$\widetilde{X}$ and $M$. Thus the $\R^2$ factor of $\widetilde{X}$ 
has coordinates $x, y$ coming from $M$. Let $\pi_x$, $\pi_y \co
\widetilde{X}= \R^2 \times T \to \R^2$ be the projection maps onto the
$x$- and $y$-axes: $\pi_x(x,y,t) = (x,0)$ and $\pi_y(x,y,t) = (0, y)$. 

Given $t\in T$ and a subset $S \subset \R^2 \times t$, let $\ellx(S)$ be
the length of $\pi_x(S) \times h_0(t)$ considered as a subset of
$M$. This subset is contained in a line parallel to the $x$-axis, and its
length in $M$ will depend on the height of $t$. Similarly, let $\elly(S)$
be the length of $\pi_y(S) \times h_0(t)$. Since the metric on $\R^2
\times t$ is Euclidean, we have 
\begin{equation}\label{productarea} 
\area(S) \ \leq \ \ellx(S) \, \elly(S). 
\end{equation}

Now consider two additional projection maps in $M$: the map $\Pi_x \co M
\to M$ given by $(x, y, z) \mapsto (x, 0, z)$, and $\Pi_y\co M \to M$
given by $(x, y, z) \mapsto (0, y, z)$. If we consider the image
coordinate planes in their induced metrics, both of these maps are
area-decreasing for surfaces in $M$. 

We wish to estimate the area of $(F_i)_{\text{\sl high}}$ using
equation~\eqref{productarea}. For this, we will relate
$\ellx((F_i)_{\text{\sl high}})$ and $\elly((F_i)_{\text{\sl high}})$ to
the areas of $A_i$ and $B_i$. Consider two more families of sets in $M =
\R^2 \times \R$ :
\begin{align*}
Q_i \ &= \ \pi_x((F_i)_{\text{\sl high}}) \times (h_i, h_i + 1), \\
R_i \ &= \ \pi_y((F_i)_{\text{\sl high}}) \times (h_i+ L, h_i + L + 1).
\end{align*}
These sets are contained in the $xz$- and $yz$-coordinate planes
respectively, and their areas may be measured in the induced (hyperbolic)
metrics. 
\end{noname}

\begin{lemma}\label{QRareas} 
For each $i$ we have 
\begin{itemize}
\item[(a)] $\ellx((F_i)_{\text{\sl high}}) \leq \lambda \area(Q_i)$ 
\item[(b)] $\elly((F_i)_{\text{\sl high}}) \leq \mu^L \area(R_i)$. 
\end{itemize}
\end{lemma}

\begin{proof}
For (a), the induced metric on the $xz$-coordinate plane is given by
$ds^2 = \lambda^{-2z} dx^2 + dz^2$, with area element $\lambda^{-z}
dx \, dz$. Let $D_i \subset \R$ be the projection $\{x \in \R \mid (x,0) \in
\pi_x((F_i)_{\text{\sl high}}) \}$. We have 
\begin{align*}
\area(Q_i) \ = \ \int_{D_i} \int_{h_i}^{h_i + 1} \lambda^{-z} dz \, dx
\ &\geq \ \int_{D_i} \int_{h_i}^{h_i + 1} \lambda^{-h_i-1} dz \, dx \\
&= \ \lambda^{-1} \int_{D_i} \lambda^{-h_i} dx \ = \ \lambda^{-1}
\ellx((F_i)_{\text{\sl high}}). 
\end{align*}
The inequality holds since $\lambda > 1$, and the last equality holds
since $F_i$ has height $h_i$. 

Part (b) is similar. The $yz$-plane has metric given by $ds^2
= \mu^{-2z} dy^2 + dz^2$ with area element $\mu^{-z}dy \, dz$. Let
$E_i\subset \R$ be the projection $\{y \in \R \mid (0,y) \in
\pi_y((F_i)_{\text{\sl high}}) \}$. Then 
\begin{align*}
\area(R_i) \ = \ \int_{E_i} \int_{h_i+L}^{h_i+L+1} \mu^{-z} dz \, dy \ &\geq
\ \int_{E_i} \int_{h_i+L}^{h_i+L+1} \mu^{-h_i-L} dz \, dy \\
&= \ \mu^{-L} \int_{E_i} \mu^{-h_i} dy \ = \ \mu^{-L}
\elly((F_i)_{\text{\sl high}}). 
\end{align*}
This time, the inequality holds because $\mu < 1$. 
\end{proof}

\begin{proposition}\label{highpart} 
$\area(F_{\text{\sl high}}) \leq \lambda \area(\partial W)^{2 +
\log_{\lambda}(\mu)}$. 
\end{proposition}

\begin{proof}
We will show that 
\begin{equation}\label{highparteqn}
\area((F_i)_{\text{\sl high}}) \ \leq \ \lambda \mu^L \area(A_i)
\area(B_i) 
\end{equation}
for all $i$. Then, summing over $i$ and applying Lemma~\ref{disjoint}, we
obtain 
\begin{equation*}
\area(F_{\text{\sl high}}) \ \leq \ \lambda \mu^L \area(\partial W)^2 
\end{equation*}
which implies the proposition by equation~\eqref{muL}. 

To establish \eqref{highparteqn} it suffices to show that $\area(Q_i)
\leq \area(A_i)$ and $\area(R_i) \leq \area(B_i)$ and to apply
equation~\eqref{productarea} and Lemma~\ref{QRareas}. 

First we claim that $\Pi_y (q(B_i))$ contains $R_i$. Choose any $p\in
(F_i)_{\text{\sl high}}$ and $h \in (h_i + L, h_i + L + 1)$. Write $p$ as
$(p_0, t_0) \in \R^2 \times T$ and $\pi_+(p)$ as $(p_0, t_1)$. The
segment $p_0 \times [t_0, t_1]$ is part of the fiber $\pi_-^{-1}(p)$, and is
contained in $W$. Since $p$ is in the high part of $F_i$, the height of
$t_1$ is at least $h_i + L + 1$, and there is a unique $t \in [t_0, t_1]$
of height $h$. Now we have $(p_0, t) \in W$. The line through $(p_0, t)$
parallel to the $x$-axis must exit $W$, at some point $b\in B_i$. Now
$\Pi_y(q(b)) = (\pi_y(b), h) = (\pi_y(p), h)$, and we have shown that
$R_i \subset \Pi_y(q(B_i))$. 

By a similar argument, $\Pi_x(q(A_i))$ contains $Q_i$ (reverse the
roles of $x$ and $y$ and choose $h \in (h_i, h_i + 1)$). Now recall that
$\Pi_x$ and $\Pi_y$ are area-decreasing and $q$ is locally isometric. It
follows that $\area(B_i) \geq \area(R_i)$ and $\area(A_i) \geq
\area(Q_i)$, as needed. 
\end{proof}

Finally, putting together Propositions \ref{foldbound}, \ref{lowpart}, and
\ref{highpart}, and consolidating constants (with the assumption that
$\area(\partial W) \geq 1$), we obtain 
\begin{equation}
\RVol(W) \ \leq \ \left(\frac{2\lambda(\mu +1) +1}{\ln(\lambda \mu)}\right)
\area(\partial W)^{2 + \log_{\lambda}(\mu)}  
\end{equation}
which has the form of the desired upper bound for $\Delta^{(2)}(x)$.

\subsection*{The general case} Now we return to the situation given at 
the beginning of this section, where $g \co W \to \widetilde{X}$ is a
least-volume transverse extension of $f \co \partial W \to
\widetilde{X}^{(2)}$. The proof will follow the same general outline as
in the embedded case, and we will work with analogues of the various
items $F_i$, $A_i$, $B_i$, $Q_i$, $R_i$, etc. The proof itself does not
depend formally on the embedded case, though we will use several of the
intermediate results obtained thus far. 

\begin{noname}
We need to introduce some terminology related to the generalized handle 
decomposition of $W$. Recall that a $2$-cell of $\widetilde{X}$ is either
horizontal or vertical, accordingly as it maps to a vertex or an edge of
the tree $T$. 

A $1$-handle is \emph{horizontal} if it maps to a horizontal $2$-cell of
$\widetilde{X}$ and is not a floating $1$-handle (i.e. it is homeomorphic
to $I\times D^{2}$, and not to $S^{1}\times D^{2}$). A $1$-handle is
\emph{vertical} if it maps to a vertical $2$-cell of $\widetilde{X}$ and
is not a floating $1$-handle. Thus, every $1$-handle is either
horizontal, vertical, or is a floating handle. 
\end{noname}

\begin{remark}\label{0handlesreduced}
Every non-floating $1$-handle either joins a $0$-handle to a $0$-handle,
a $0$-handle to $\partial W$, or $\partial W$ to $\partial W$. In the
first case, since the map $g$ is least-volume, the two $0$-handles map to
\emph{distinct} $3$-cells of $\widetilde{X}$. For otherwise, the two
neighboring $0$-handles can be cancelled by the procedure described in
Section~\ref{prelimsection}, reducing the volume of $g$. No $1$-handle
joins a $0$-handle to itself, since $\widetilde{X}$ has the property that
no $2$-cell appears more than once as a ``face'' of any single $3$-cell;
the closure of a $3$-cell in $\widetilde{X}$ is an embedded ball with
interior equal to the open $3$-cell. 
\end{remark}

\begin{noname}
We will need to make use of some vector fields on $W$, obtained by
pulling back the coordinate vector fields on $M$ via the map $q \circ g
\co W \to M$. These vector fields will be denoted
$\frac{\partial}{\partial x}$, $\frac{\partial}{\partial y}$, and
$\frac{\partial}{\partial z}$, and they are defined on the interiors of
the $0$-handles. In particular, every $0$-handle has an ``upward''
direction given by $\frac{\partial}{\partial z}$. 

We say that a horizontal 1-handle $H$ is \emph{minimal} if
$\frac{\partial}{\partial z}$ is directed \emph{away} from $H$ in both
neighboring $0$-handles. Such a 1-handle is a local minimum for the
height function (the $z$-coordinate) on the tree $T$. 

Since $T$ branches only in the upward direction, and since horizontal
$1$-handles are joined to $0$-handles mapping to distinct $3$-cells in
$\widetilde{X}$, there are no ``maximal'' 1-handles $H$ (where
$\frac{\partial}{\partial z}$ is directed toward $H$ on both ends). Hence
if a horizontal handle $H = I \times D^2$ is not minimal, then
$\frac{\partial}{\partial z}$ on the neighboring $0$-handles can be
extended to a non-vanishing vector field on $H$, tangent to the $I$
factor. Thus we will always regard $\frac{\partial}{\partial z}$ as being
defined (and non-zero) on the union of the $0$-handles and the
non-minimal horizontal $1$-handles. 

Let $\Fz$ be the partial foliation on $W$ 
whose leaves are the orbits of the flow along
$\frac{\partial}{\partial z}$. Some leaves of $\Fz$ may terminate or
originate in a $2$- or $3$-handle of $W$. These are the leaves whose
images in $\widetilde{X}$ meet a $0$- or $1$-cell. In terms of transverse
area, the set of such leaves has measure zero, and we will discard them
from $\Fz$. Note that the remaining leaves of $\Fz$ still meet the
$0$-handles in a set of full measure. Let $U_z$ denote the union of the
leaves of $\Fz$.

\medskip
Every vertical $2$-cell of $\widetilde{X}$ is a face of exactly two
$3$-cells, and also is not tangent to the vector fields
$\frac{\partial}{\partial x}$ or $\frac{\partial}{\partial y}$. (The
sides of $Q$ are not parallel to the $x$- or $y$-axes because the matrix
$A$ is hyperbolic.) These facts, together with Remark
\ref{0handlesreduced}, imply that for any vertical $1$-handle $H 
= I \times D^2$, the vector field $\frac{\partial}{\partial x}$ on the
neighboring $0$-handles extends to a non-vanishing vector field on $H$,
tangent to the $I$ factor. By adjusting lengths, we can arrange that this
field is independent of the $z$-coordinate (this is already true in the
$0$-handles). The vector field $\frac{\partial}{\partial y}$ is defined
similarly. We also define partial foliations $\Fx$ and $\Fy$ on the union
of the $0$-handles and vertical $1$-handles, analogously to $\Fz$. Note
that these two foliations coincide in the vertical $1$-handles, even
though they are transverse elsewhere. Again, we will discard all leaves
terminating or originating in a $2$- or $3$-handle of $W$. Let $U_x$ and
$U_y$ denote, respectively, the unions of the leaves of $\Fx$ and of
$\Fy$. 
\end{noname}

\begin{noname}
Every leaf of $\Fz$ is homeomorphic to $\R$ and is oriented by the vector
field $\frac{\partial}{\partial z}$. It terminates in a well-defined
point of $\partial W$, and originates either at a point in $\partial W$
or at a point in the boundary of a minimal $1$-handle. Similarly, every
leaf of $\Fx$ and $\Fy$ both originates and terminates on $\partial
W$. For $p \in U_{\alpha}$ let $\tau_{\alpha}(p)$ denote the terminal
point of the leaf of $\Fs_{\!\alpha}$ containing $p$ (for $\alpha = x, y,
z$). This defines maps $\tau_{\alpha}\co U_{\alpha} \to \partial W$. 
Also let $o_{\alpha}(p)$ be the origination point of the leaf of
$\Fs_{\!\alpha}$ containing $p$. 
\end{noname}

\begin{definition}
We wish to define the \emph{fold sets} in $W$, which will be embedded
surfaces with boundary (minus a measure zero set). 
Let $e_1, \ldots, e_k$ be the closed edges of 
$T$ which meet the image of $\pi_T \circ g$. Given $e_i$ and a point
$p_i$ in the interior of $e_i$, the preimage $(\pi_T \circ g)^{-1}(p_i)$
is a properly embedded surface $\Sigma_i \subset W$, by transversality,
and the preimage of the interior of $e_i$ is an open regular
neighborhood of $\Sigma_i$. The intersection of $\Sigma_i$ with the
handle decomoposition of $W$ is a handle decomposition of $\Sigma_i$,
and the map is transverse with respect to this structure. The closure
of the preimage of the interior of $e_i$ is a union of handles of $W$,
and is a codimension-zero submanifold of $W$, homeomorphic to $\Sigma_i
\times I$, with the product handle structure. That is, each $0$-, $1$-, or
$2$-handle of $\Sigma_i \times I$ is the product of a $0$-, $1$-, or
$2$-handle of $\Sigma_i$ with $I$. The product structure $\Sigma_i \times
I$ is chosen so that fibers $p \times I$ map by $q \circ g$ into vertical
lines in $M$ (in particular, $I$ corresponds to the $z$-coordinate in the
$0$-handles). 

Let $v_i$ be the lower endpoint of $e_i$ (with respect to the height
function), and orient the $I$ factor of $\Sigma_i \times I$ so that
$\Sigma_i \times 0$ maps to $v_i$. The handles of $W$ comprising
$\Sigma_i \times I$ are all $0$-, $1$-, and $2$-handles. Various $1$-,
$2$-, and $3$-handles (those mapping to $v_i$ by $\pi_T \circ g$) may be
attached in part to $\Sigma_i \times 0$. Let $E_i$ be the intersection of
$\Sigma_i \times 0$ with the union of all minimal $1$-handles. It is a
codimension-zero submanifold of $\Sigma_i \times 0$, equal to a union of 
attaching regions of minimal $1$-handles. Every minimal $1$-handle is
attached to two surfaces $E_i, E_j$ for some $i \not= j$, since the
adjacent $0$-handles are distinct and map to distinct edges of
$T$. Lastly, define $F_i$ to be $E_i \cap U_z$. Note that $F_i$ has full
measure in $E_i$.  
%
%
\end{definition}

Having defined $F_i$ and $v_i$, note that various vertices $v_i$
may now coincide (unlike the embedded case). Define the heights
$h_i$ exactly as before: $h_i = h_0(v_i)$. Define $L =
\log_{\lambda}(\RArea(f))$, and note that equations analogous to
\eqref{lambdaL}--\eqref{lambdamuL} hold:  
\begin{align}
\lambda^L \ &= \ \RArea(f), \label{ilambdaL} \\
\mu^L \ &= \ \RArea(f)^{\log_{\lambda}(\mu)}, \label{imuL} \\
(\lambda \mu)^L \ &= \ \RArea(f)^{1 +
\log_{\lambda}(\mu)}. \label{ilambdamuL} 
\end{align}

We redefine the subtrees $T_i$ to be smaller
than those from section~\ref{items}, by splitting along the edges above
the vertex. That is, we now define
\[T_i \ = \ \{x \in T \mid \interior(e_i) \cap [x,\eta)
\not= \emptyset\}.\] 
This is an \emph{open} subtree of $T$, not containing $v_i$. Define
$S_i$, $A_i$, and $B_i$ as follows: 
\begin{itemize}
\item $S_i = \partial W \cap \closure((g \circ \pi_T)^{-1} (T_i))$, 
\item $A_i = S_i \cap (g \circ h)^{-1}((h_i, h_i + 1))$, 
\item $B_i = S_i \cap (g \circ h)^{-1}((h_i + L, h_i + L + 1))$.
\end{itemize}
Note that $S_i$ is a subsurface of $\partial W$ and $\partial S_i =
\partial W \cap (\Sigma_i \times 0)$. The next lemma has essentially the
same proof as Lemma~\ref{disjoint}. 

\begin{lemma}\label{disjoint2}
$A_i \cap A_j = B_i \cap B_j = \emptyset$ for $i\not= j$. \qed
\end{lemma}

Now let $F = \bigcup_i F_i$, and define $\pi_+ \co F 
\to \partial W$ to be the restriction $\tau_z \vert_{F}$. That
is, $\pi_+$ flows $F$ ``upward'' along $\frac{\partial}{\partial z}$ to
$\partial W$. Note that $\pi_+$ is indeed defined on $F$, and is
injective. Define the \emph{low} and \emph{high} parts of $F$ as before: 
\begin{align*}
(F_i)_{\text{\sl low}} \ &= \ \{p \in F_i \mid h(g(\pi_+(p))) \leq h_i + L
+ 1\}, \\ 
(F_i)_{\text{\sl high}} \ &= \ \{p \in F_i \mid h(g(\pi_+(p))) \geq h_i +
L + 1\}.  
\end{align*}
Also define $F_{\text{\sl low}} = \bigcup_i (F_i)_{\text{\sl low}}$ and
$F_{\text{\sl high}}= \bigcup_i (F_i)_{\text{\sl high}}$. 

\begin{lemma}\label{foldbound2}
$\RVol(g) \ \leq \ \frac{1}{\ln(\lambda \mu)} ( \RArea(f) + \RArea(g
\vert_F))$. 
\end{lemma}

\begin{proof}
We have $\RVol(g) = \RVol(g \vert_{U_z})$ since $U_z$ has full measure in
the $0$-handles of $W$. Note that every leaf of $\Fs_z$ starts on $F$ or
on $\partial W$, and ends in $\partial W$. Thus we may decompose $U_z$ as
$U_z^F \cup U_z^{\partial}$ where 
\begin{align*}
U_z^F \ &= \ \{ p \in U_z \mid o_z(p)\in F \, \}, \\
U_z^{\partial} \ &= \ \{ p \in U_z \mid o_z(p) \in \partial W \, \}. 
\end{align*}
Now $\RVol(g \vert_{U_z}) = \RVol(g \vert_{U_z^F}) + \RVol(g
\vert_{U_z^{\partial}})$. By pulling back the metric from $\widetilde{X}$
and integrating along leaves of $\Fs_z$, we have 
\[ \RVol(g \vert_{U_z^{F}}) \ \leq \ \frac{1}{\ln(\lambda \mu)} \RArea(g
  \vert_F) \]
and 
\[ \RVol(g \vert_{U_z^{\partial}}) \ \leq \ \frac{1}{\ln(\lambda \mu)}
\RArea(g \vert_{\partial W}) \ =  \ \frac{1}{\ln(\lambda
\mu)}\RArea(f). \qedhere\] 
\end{proof}

\begin{remark} 
In the current situation, there is no ambiguity or choice involved in the
definition of $\pi_+$. The difference with the embedded case is that 
each minimal $1$-handle has \emph{two} attaching regions contributing to
$F$, and there is a unique way to flow upward from each side. 
In effect, the fold set has been doubled, and this also accounts for the
missing factor of $2$ in Lemma \ref{foldbound2} (compared with
Proposition \ref{foldbound}). 
\end{remark}

Our main task now is to bound $\RArea(g \vert_F)$ in terms of
$\RArea(f)$. The next result is entirely analogous to Proposition
\ref{lowpart}, and has the same proof. The only difference is that here
the area elements are pulled back from $\widetilde{X}$. 

\begin{proposition}\label{lowpart2}
$\RArea(g \vert_{F_{\text{\sl low}}}) \ \leq \ (\lambda \mu) \RArea(f)^{2
+ \log_{\lambda}(\mu)}$. \qed 
\end{proposition}

Next we need an analogue of equation \eqref{productarea}. 
In order to define the lengths $\ell_x$ and $\ell_y$ for the sets 
$(F_i)_{\text{\sl high}}$, we need to extend the vector fields
$\frac{\partial}{\partial x}$ and $\frac{\partial}{\partial y}$ to the
surfaces $\Sigma_i \times 0$. 
Recall that $\Sigma_i \times I$ has a product handle structure, 
and these
vector fields are defined in the interiors of the $0$-handles and
$1$-handles (all of which are vertical). 
Note that $\frac{\partial}{\partial x}$, in
the interior of $\Sigma_i \times I$, is zero in the $I$ factor and
constant (as $t \in I$ is varied) in the $\Sigma_i$ factor. Thus
$\frac{\partial}{\partial x}$ extends continuously to $\Sigma_i \times 0$
as a non-vanishing field, defined on the interiors of the $0$- and
$1$-handles of $\Sigma_i\times 0$. Any leaf of $\Fs_x$ meeting  $\Sigma_i
\times 0$ remains entirely within $\Sigma_i \times 0$, since
$\frac{\partial}{\partial x}$ is tangent to this surface (indeed, every
$\Sigma_i\times t$ has this property). 
The vector field $\frac{\partial}{\partial y}$ extends to $\Sigma_i
\times 0$ in the same way. Lastly, we discard leaves of $\Fs_x$ and
$\Fs_y$ meeting $2$-handles of $\Sigma_i \times 0$, so that every
leaf in $\Sigma_i \times 0$ begins and ends in $\partial S_i$. These
remaining leaves have full measure in the $0$-handles of $\Sigma_i \times
0$. 

We now define $\ell_x((F_i)_{\text{\sl high}})$ to be the transverse
measure of the set of leaves of $\Fs_y$ meeting $(F_i)_{\text{\sl
high}}$. That is, we project $(F_i)_{\text{\sl high}} \cap U_y$ to
$\partial S_i$ using $\tau_y$, and then measure this set by integrating the
pullback of the length element $\lambda^{-z} dx$ from $M$. Similarly, 
$\ell_y((F_i)_{\text{\sl high}})$ is defined using the 
length element $\mu^{-z}dy$. 

\begin{proposition}\label{productarea2} 
$\RArea(g \vert_{(F_i)_{\text{\sl high}}}) \ \leq \
\ell_x((F_i)_{\text{\sl high}}) \, \ell_y((F_i)_{\text{\sl high}})$ for each
$i$. 
\end{proposition}

\begin{proof}
First observe that the intersection of a leaf of $\Fs_x$ and a leaf of
$\Fs_y$ is either one point (in a $0$-handle of $\Sigma_i \times 0$), a
closed interval (in a $1$-handle of $\Sigma_i \times 0$), or is empty. 
To see this, map both leaves to $M$ and project onto the $x$-axis. Each
$\Fs_y$ leaf maps to a single point, whereas each $\Fs_x$ leaf maps
monotonically, with point preimages equal to sets of the form described
above. 

It follows that the map 
\[ \tau_y \times \tau_x \co (\Sigma_i \times 0) \cap U_x \cap U_y
\ \to \ \partial S_i \times \partial S_i \]
is injective when restricted to the $0$-handles of $\Sigma_i \times 0$. 

Next define the map ${g}_i \co \Sigma_i \times 0 \to \R^2$ to
be $q \circ g \co \Sigma_i \times 0 \to M$ followed by projection
onto the first two coordinates of $M = \R^3$. Thus, $q(g(p)) =
({g}_i(p), h_i) \in M$ for all $p \in \Sigma_i \times 0$. Let
$\pi_x, \pi_y \co \R^2 \to \R$ be projections onto the first 
and second coordinates respectively. It is easily verified that
${g}_i$ agrees with the following composition of maps: 
\[ (\Sigma_i \times 0) \cap U_x \cap U_y \ \xrightarrow{\tau_y \times
\tau_x} \ \partial S_i 
\times \partial S_i \ \xrightarrow{{g}_i \times {g}_i} \ \R^2 \times
\R^2 \ \xrightarrow{\pi_x \times \pi_y} \ \R \times \R.\]
(Write $q(g(p))$ as $(x_p, y_p, h_i)$; both maps send $p$ to $(x_p,
y_p)$.) 

Recall that $\Sigma_i \times 0$ maps into $\R^2 \times h_i \subset M$,
and so the surface area element being pulled back in the computation of
$\RArea(g\vert_{(F_i)_{\text{\sl high}}})$ is the horizontal area
element of $M$. This element is just the product of the length
elements $\lambda^{-z} dx$ and $\mu^{-z} dy$. 

In the integrals below, $(F_i)_{\text{\sl high}}$ is understood to be
restricted to the $0$-handles of $\Sigma_i \times 0$ (where area is
supported). We have 
\begin{align*}
\RArea(g\vert_{(F_i)_{\text{\sl high}}}) \ &= \ \int_{(F_i)_{\text{\sl
high}}} (q \circ g)^*(\lambda^{-z}dx \, \mu^{-z}dy) \\
&= \ \int_{(F_i)_{\text{\sl high}} \cap U_x \cap U_y} (\pi_x \times \pi_y
\circ {g}_i \times {g}_i \circ \tau_y \times \tau_x)^*(\lambda^{-z}dx \, \mu^{-z}dy) 
\end{align*}
which, by injectivity of $\tau_y \times \tau_x$, is at most 
\[\int_{\tau_y((F_i)_{\text{\sl high}} \cap U_x \cap U_y) \times
\tau_x((F_i)_{\text{\sl high}} \cap U_x \cap U_y)} (\pi_x \times \pi_y
\circ {g}_i \times {g}_i)^* (\lambda^{-z}dx \, \mu^{-z}dy).\]
The latter is equal to 
\[\int_{\tau_y((F_i)_{\text{\sl high}} \cap U_x \cap U_y)} (\pi_x
\circ {g}_i)^*(\lambda^{-z} dx) \ 
\int_{\tau_x((F_i)_{\text{\sl high}} \cap U_x \cap U_y)} (\pi_y \circ
{g}_i)^*(\mu^{-z} dy),\]
which is just $\ell_x((F_i)_{\text{\sl high}}) \
\ell_y((F_i)_{\text{\sl high}})$. 
\end{proof}

In \ref{furtherdefs} we defined the projection maps $\Pi_x, \Pi_y \co M
\to M$, sending $(x, y, z)$ to the points $(x, 0, z)$ and $(0, y, z)$ 
respectively. We also had projections $\pi_x, \pi_y \co \widetilde{X} =
\R^2 \times T \to \R^2$, mapping $(x,y, t)$ to $(x,0)$ and $(0,y)$
respectively. Define the sets $Q_i$, $R_i \subset M = \R^2 \times \R$ as
follows: 
\begin{align*}
Q_i \ &= \ \pi_x(g((F_i)_{\text{\sl high}}
)) \times (h_i, h_i + 1), \\
R_i \ &= \ \pi_y(g((F_i)_{\text{\sl high}}
)) \times (h_i+ L, h_i + L + 1).
\end{align*}
The claims of Lemma \ref{QRareas} remain true exactly as stated, and
are proved in the same way. 
Thus:

\begin{lemma}\label{QRareas2}
For each $i$ we have 
\begin{itemize}
\item[(a)] $\ellx((F_i)_{\text{\sl high}}) \leq \lambda \area(Q_i)$ 
\item[(b)] $\elly((F_i)_{\text{\sl high}}) \leq \mu^L \area(R_i)$. \qed
\end{itemize}
\end{lemma}

Next we adapt Proposition \ref{highpart} to the current situation. 

\begin{proposition}\label{highpart2} 
$\RArea(g \vert_{F_{\text{\sl high}}}) \ \leq \ \lambda
\RArea(f)^{2+\log_{\lambda}(\mu)}$. 
\end{proposition}

\begin{proof}
As in the proof of Proposition \ref{highpart}, it suffices to show that
$\area(Q_i) \ \leq \ \RArea(f\vert_{A_i})$ and $\area(R_i) \ \leq \
\RArea(f\vert_{B_i})$ for each $i$:
since 
\[\RArea(g\vert_{(F_i)_{\text{\sl high}}}) \ \leq \ \lambda \mu^L
\area(A_i) \area(B_i)\] 
by Proposition \ref{productarea2} and Lemma 
\ref{QRareas2}, we then have 
\[\RArea(g\vert_{(F_i)_{\text{\sl high}}}) \ \leq \ \lambda \mu^L
\RArea(f\vert_{A_i}) \RArea(f\vert_{B_i})\] 
for all $i$. Summing over $i$, using Lemma \ref{disjoint2}, we obtain the
desired inequality, by \eqref{imuL}. 

We claim that $\Pi_y(q(f(B_i)))$ contains a subset of $R_i$ of full
measure. Given a point in $R_i$, it is determined by points $p \in
(F_i)_{\text{\sl high}}$ and $h \in (h_i+L, h_i+L + 
1)$. Let $p' \in W$ be a point on the leaf of $\Fs_z$ through $p$ of
height $h$; such a point exists since $p$ has height $h_i$ and $\pi_+(p)$
has height $h_i+L+1$ or greater. Write $q(g(p'))$ as $(x_{p'}, y_{p'},
h)$ in the coordinates of $M$, and note that $q(g(p)) = (x_{p'}, y_{p'},
h_i)$. Thus $\pi_y(g(p)) = (0, y_{p'})$. 

If $p' \in U_x$ then $\tau_x(p')$ is defined and is in $B_i$, and
\[\Pi_y(q(f(\tau_x(p')))) \ = \ (0, y_{p'}, h) \ = \ (\pi_y(g(p)), h).\]
Therefore this point of $R_i$ is indeed in the image of $B_i$ under
$\Pi_y \circ q \circ f$. Thus we want to verify that $p' \in U_x$ for
almost all choices of $(\pi_y(g(p)), h) \in R_i$. 

Let $R_i'$ be the set of pairs $(\pi_y(g(p)), h) \in R_i$ such that $h$
is not an integer. 
Let $K \subset \widetilde{X}$ be the intersection of $g(W)$ with the
$1$-skeleton of $\widetilde{X}$. It is a finite graph, and its image
$\Pi_y(q(K))$ has measure zero in the $yz$-plane in $M$. Note also that
all $2$- and $3$-handles of $W$ map by $g$ into $K$. 

The point $p'$ must be in the interior of a $0$-handle or a horizontal
$1$-handle of $W$, since $p'\in U_z$. In the latter case, $p'$ maps to a
horizontal $2$-cell of $\widetilde{X}$, and so $h$ is an
integer. In the former case, $\frac{\partial}{\partial x}$ is defined at
$p'$. If $p'\not\in U_x$ then the (discarded) leaf of $\Fs_x$ through
$p'$ meets a $2$- or $3$-handle. Then $\Pi_y(q(g(p')))$ is contained in
the measure zero set $\Pi_y(q(K))$. But $\Pi_y(q(g(p')))$ is the original
point $(\pi_y(g(p)),h) \in R_i$. The argument above therefore shows that
$\Pi_y(q(f(B_i)))$ contains $R_i' - \Pi_y(q(K))$, a subset of $R_i$ of
full measure. 

Thus $\area\Pi_y(q(f(B_i))) \geq \area(R_i)$. 
Since $\Pi_y$ is area-decreasing and $q$ locally isometric, we conclude
that $\RArea(f\vert_{B_i}) \geq \area(R_i)$. By a similar argument,
$\RArea(f\vert_{A_i}) \geq \area(Q_i)$. 
\end{proof}

\subsection*{The bound}
We can now determine an upper bound for $\Delta^{(2)}(x)$. Assembling
Lemma \ref{foldbound2} and Propositions \ref{lowpart2}, \ref{highpart2} and
consolidating constants, we find that 
\begin{equation}\label{finalbound} 
\RVol(g) \ \leq \ \left(\frac{1 + \lambda(\mu + 1)}{\ln(\lambda
\mu)}\right) \RArea(f)^{2 + \log_{\lambda}(\mu)}. 
\end{equation}
Recall that all $3$-cells of $\widetilde{X}$ have the same volume $V$ (and
hence $\Vol^3(g) = \frac{1}{V}\RVol(g)$). Let $C$ be the largest 
Riemannian area of a $2$-cell of $\widetilde{X}$ (or equivalently, of
$X$). Then $\RArea(f) \leq C \Vol^2(f)$, and by \eqref{finalbound} we have 
\begin{equation*}
\Vol^3(g) \ \leq \ \left(\frac{1 + \lambda(\mu + 1)}{V \ln(\lambda
\mu)}\right) (C \Vol^2(f))^{2 + \log_{\lambda}(\mu)}. 
\end{equation*}
Therefore $\FVol^W(f) \leq D (\Vol^2(f))^{2 + \log_{\lambda}(\mu)}$ for a 
constant $D$ depending only on the original matrix $A$ (which determined
$\lambda$, $\mu$, and the geometry of $\widetilde{X}$). Since the
$3$-manifold $W$ was arbitrary, we have now established that
$\Delta^{(2)}(x) \leq D x^{2 + \log_{\lambda}(\mu)}$, and therefore
$\delta^{(2)}(x) \preceq \Delta^{(2)}(x) \preceq x^{2 + \log_{\lambda}(\mu)}$.

\section{The lower bound}\label{lowersec}

To establish a lower bound for $\delta^{(2)}(x)$ we want a sequence
of embedded balls $B_n \subset \widetilde{X}$ whose volume growth is as
large as possible, relative to the growth of boundary area. The optimal
shape is a ball made from two half-balls, each contained in a copy of $M$
inside $\widetilde{X}$, joined along their bottom faces. The half-balls
in $M$ will need to have large volume compared to ``upper'' boundary
area. 

For the half-balls, we begin by defining optimally proportioned regions
$R_n \subset M$, which are easy to measure in the Riemannian
metric. Then we approximate these regions combinatorially by subcomplexes
$S_n$.

\subsection*{Extremal Riemannian regions} 
In the coordinates of $M$, define
\[R_n \ = \ [0, \lambda^n] \times [0, (\lambda \mu)^n] \times [0,n].\] 
The volume of $R_n$ is easily computed by integration. Each horizontal
slice $[0, \lambda^n] \times [0, (\lambda \mu)^n] \times z$ has area
$\lambda^n (\lambda \mu)^n (\lambda \mu)^{-z}$, and integrating in the
$z$-coordinate yields 
\begin{equation} \label{volRn}
\RVol(R_n) \ = \ \frac{1}{\ln(\lambda \mu)} (\lambda^n(\lambda \mu)^n -
\lambda^n). 
\end{equation}
Recall that $\lambda\mu = \det(A) \geq 2$. If $n \geq 1$ then
$\frac{1}{2}(\lambda \mu)^n \geq 1$, whence $(\lambda \mu)^n - 1 
\geq \frac{1}{2}(\lambda \mu)^n$. Together with \eqref{volRn} this
implies 
\begin{align} \label{volRn2}
\RVol(R_n) \ &\geq \ \frac{1}{2\ln(\lambda \mu)} \lambda^n (\lambda
\mu)^n \notag \\
\ &= \ \frac{1}{2\ln(\lambda \mu)}\bigl(\lambda^n\bigr)^{2 +
\log_{\lambda}(\mu)} 
\end{align}
for $n \geq 1$. 

Next we consider the areas of the various faces of $R_n$. The top face
has area $\lambda^n$ (taking $z = n$, above). Next, the segment $[0,
\lambda^n] \times y \times z$ has length $\lambda^n
\lambda^{-z}$. Integrating with respect to $z$, we find that the faces
$[0,\lambda^n] \times 0 \times [0,n]$ and $[0,\lambda^n] \times
(\lambda\mu)^n \times [0,n]$ each have area
$\frac{1}{\ln(\lambda)}(\lambda^n - 1)$. By a similar computation, the
other two vertical faces each have area $\frac{1}{\ln(\mu)}\lambda^n
(\mu^n - 1) = \frac{1}{\ln(\mu^{-1})}\lambda^n (1 - \mu^n)$. Since $\mu <
1$, this quantity is less than $\frac{1}{\ln(\mu^{-1})}\lambda^n$. Now
let $\partial^+ R_n$ denote the union of the five faces (omitting the
bottom face) of $R_n$. We have shown that 
\begin{equation}\label{areaRn}
\RArea(\partial^+ R_n) \ \leq \ \bigl(1 + (2/\ln\lambda) -
(2/\ln\mu)\bigr) \lambda^n. 
\end{equation}

\subsection*{Extremal combinatorial regions}
Recall that $D$ is the matrix $BAB^{-1} = \big( \, {}^{\lambda}_0 \
{}^0_{\mu} \, \big)$, and $\Gamma$ is the lattice $B(\Z \times
\Z)$, preserved by $D$. 
Fix any standard copy of $M$ inside $\widetilde{X}$, corresponding to a
line $L \subset T$. Then $M$ is a subcomplex of $\widetilde{X}$, and we
need to understand its cell structure. Note that $M$ is a union of
subcomplexes $\R^2 \times [i-1,i]$ for $i \in \Z$. Consider the
subcomplex $\R^2 \times [0,1]$. Possibly after a horizontal translation,
the closed $3$-cells are the sets $\gamma(Q) \times [0,1]$, for $\gamma$
in $\Gamma$ (recall that $Q$ is a fundamental domain for $\Gamma$ acting
on $\R^2$). Figure
\ref{fig:Q} shows the top and bottom faces of one of these $3$-cells, in
the case of no translation. 

To be more specific, let $\Gamma'$ be the lattice $D^{-1}(\Gamma)$, and
note that $\Gamma'$ contains $\Gamma$ as a subgroup of index $d$. Then the
$3$-cells of $\R^2\times [0,1]$ are the sets $\gamma(Q) \times [0,1]$
where $\gamma$ ranges over a single coset of $\Gamma$ in $\Gamma'$. 

Continuing upward, the closed $3$-cells of $\R^2 \times [i-1, i]$
are the sets $\gamma(D^{i-1}(Q)) \times [i-1, i]$, where $\gamma$ ranges
over a coset of $D^{i-1}(\Gamma)$ in $\Gamma'$. The choice of coset
depends on the path in $T$ followed by $L$ from height $0$ to height
$i$. (There are $d^i$ such paths, and cosets.) Thus, the various copies of
$M$ inside $\widetilde{X}$ have differing cell structures (with respect
to the standard coordinates), though at each height they agree up to
horizontal translation. 

For $i = 1, 2, \ldots$ let $\Lambda_i \subset \R^2$ be the union of the
sides of $\gamma(D^{i-1}(Q))$ for $\gamma$ in the appropriate coset of
$D^{i-1}(\Gamma)$ in $\Gamma'$. Then $\Lambda_i \times i$ is a subcomplex
of $M$, and in fact, so is $\Lambda_i \times [i-1, i]$. This latter
subcomplex is the smallest subcomplex containing the vertical $1$- and
$2$-cells of $\R^2 \times [i-1, i]$.

\begin{definition}
Let $w$ be the diameter of $Q$ (in $\R^2$, with the Euclidean
metric). There is a constant $k$ such that every horizontal or vertical
line segment of length $w$ intersects $\Lambda_1$ in at most $k$
points. We will call $k$ the \emph{backtracking constant} for
$\widetilde{X}$. 
\end{definition}

\begin{lemma}\label{striplemma} 
Let $W \subset \R^2$ be a region of the form $[a,a+w] \times \R$ or\/ $\R
\times [a,a+w]$. Let $\pi \co W \to \R$ be projection onto the
$\R$ factor. Then $W \cap \Lambda_1$ contains a properly embedded line
$\ell$, and the restricted map $\pi \co \ell \to \R$ is at most $k$-to-one. 
\end{lemma}

\begin{proof}
The components of $\R^2 - \Lambda_1$ are isometric copies of the interior
of $Q$. For the first statement, note that an open set of diameter $w$
cannot disconnect $W$, and so $W \cap \Lambda_1$ is connected and
contains a line joining the two ends of $W$. The second statement is
clear, since the fibers of $\pi$ are horizontal or vertical segments of
length $w$. 
\end{proof}

Applying the map $D^{i-1}$ (and possibly a translation) to Lemma
\ref{striplemma} yields the following result. Note that $D$ preserves the
horizontal and vertical foliations of $\R^2$ by lines. In particular,
$D^{i-1}$ takes fibers of $\pi$ to fibers. 

\begin{lemma}\label{striplemma2} 
Let $W \subset \R^2$ be a region of the form $[a,a+\lambda^{i-1}w] \times
\R$ or $\R \times [a,a+\mu^{i-1} w]$. Let $\pi \co W \to \R$ be
projection onto the $\R$ factor. Then $W \cap \Lambda_i$ contains a
properly embedded line $\ell$, and the restricted map $\pi \co \ell \to
\R$ is at most $k$-to-one. \qed
\end{lemma}

Now we can proceed to define subcomplexes approximating the regions
$R_n$. Given an integer $n$, we will define ``slabs'' $S_{i,n} \subset
\R^2 \times [i-1, i]$ for $i$ between $1$ and $n$. The union $\bigcup_i
S_{i,n}$ will contain $R_n$, and will have comparable volume and surface
area (the latter of which is controlled by the backtracking constant
$k$). The slabs will not fit together perfectly: there will be under- and
over-hanging portions, but the additional surface area arising in this
way is not excessive. 

Fix $n\in \Z_+$. For $i$ between $1$ and $n$, consider the four strips
\begin{align*}
W^1_i \ &= \ \R \times [-\mu^{i-1}w, 0] \\
W^2_i \ &= \ [\lambda^n, \lambda^n + \lambda^{i-1}w] \times \R \\
W^3_i \ &= \ \R \times [(\lambda\mu)^n, (\lambda\mu)^n + \mu^{i-1}w] \\
W^4_i \ &= \ [-\lambda^{i-1}w, 0] \times \R 
\end{align*}
which surround the rectangle $[0, \lambda^n] \times [0,
(\lambda\mu)^n]$. By Lemma \ref{striplemma2}, each of these strips
contains a properly embedded line in $\Lambda_i$, projecting to the $x$-
or $y$-axis in a $k$-to-one fashion, at most. Choose segments $\ell^j_i
\subset W^j_i$ in these lines which meet each other only in their
endpoints, forming an embedded quadrilateral in $\Lambda_i$ enclosing
$[0, \lambda^n] \times [0, (\lambda\mu)^n]$. Let $D_i$ be the closed region
bounded by this quadrilateral, and define the \emph{slab} $S_{i,n}$ to be
the subcomplex $D_i \times [i-1,i] \subset M$. Let $S_n = \bigcup_{i=1}^n
S_{i,n}$. 

Let $W_{i,n}$ be the rectangle delimited by the outermost sides of the
strips $W^1_i$, $W^2_i$, $W^3_i$, $W^4_i$ and note that $W_{i,n}$
contains $D_i$. The maximum width of these rectangles is $\lambda^n +
2\lambda^{n-1}w = \lambda^n(1 + 2w/\lambda)$, and the maximum height is
$(\lambda\mu)^n + 2w \leq (\lambda\mu)^n(1 + 2w)$. Let $\kappa$ be the
larger of $\log_{\lambda}(1 + 2w/\lambda)$ and $\log_{\lambda\mu}(1 +
2w)$. Then the rectangle with lower-left corner at $(- \lambda^{n-1} w, -
w)$, of width $\lambda^{n + \kappa}$ and height $(\lambda\mu)^{n +
\kappa}$, contains $W_{i,n}$ for all $i$. Let $R'_{n+\kappa}$ be $R_{n
+ \kappa}$, translated by $-\lambda^{n-1}w$ in the $x$-direction and by
$-w$ in the $y$-direction. Then we have 
\begin{equation*}
R_n \ \subset \ S_n \ \subset \ R'_{n + \kappa}. 
\end{equation*}

Let $\partial^+ S_n$ denote the largest subcomplex of the boundary of
$S_n$ which does not meet the interior of the base of $R_n$ (that is,
$(0, \lambda^n) \times (0, (\lambda\mu)^n) \times 0$).  Note that
$\partial^+ S_n$ has three parts: the \emph{top}, $D_{n}$; the
\emph{vertical part}, made of the sets $\ell^j_i \times [i-1,i]$; and the
\emph{horizontal part}, contained in the union of the annuli
$\bigl(W_{i,n} \times i\bigr) - \bigl((0, \lambda^n) \times (0, (\lambda
\mu)^n) \times i\bigr)$, for $i = 0, \ldots, n-1$. This last part
contains the horizontal $2$-cells of height $i$ in the symmetric
difference $(D_i \times i) \bigtriangleup (D_{i-1} \times i)$, where the
slabs fail to join perfectly. 

\begin{lemma}\label{areaSn}
There is a constant $C$ such that the Riemannian area of the top and
vertical parts of $\partial^+ S_n$ is at most $C \RArea(\partial^+
R'_{n+\kappa})$. 
\end{lemma}

\begin{proof}
Translating $D_n$ upward by $\kappa$, it becomes a subset of the top face
of $R'_{n+\kappa}$. Therefore its area is at most $(\lambda\mu)^{\kappa}$
times the area of the top face of $R'_{n+\kappa}$.  Next consider the
coordinate projections of $\ell^j_i \times [i-1, i]$ onto the sides of
$R'_{n+\kappa}$. These maps are at most $k$-to-one, by the construction
of $\ell^j_i$. Moreover, the Jacobians of these maps are bounded below by
some $J>0$, independent of $n$. To see this, consider for example the
coordinate projection onto the $xz$-plane (the case of odd $j$). On each
closed vertical $2$-cell the Jacobian achieves a positive minimum, and
there are finitely many such cells modulo isometries of $M$. These
isometries preserve the $xz$-plane field, and hence also the Jacobian of
this projection. The case of the $yz$-projection is similar. Now the
Riemannian area of $\bigcup_{i=1}^n \ell^j_i \times [i-1,i]$ is at most
$k/J$ times the area of one of the four sides of $R'_{n+\kappa}$ (one
side for each $j$). The result follows with $C =
\max\{(\lambda\mu)^{\kappa}, k/J\}$.
\end{proof}

\begin{lemma}\label{areaSn2}
There is a constant $D$ such that the Riemannian area of the
horizontal part of $\partial^+ S_n$ is at most $D \lambda^n$. 
\end{lemma}

\begin{proof}
Let $A_{i,n}$ be the annular region $\bigl(W_{i,n} \times i\bigr) -
\bigl( (0, \lambda^n) \times (0, (\lambda \mu)^n) \times i\bigr)$. Then
\begin{align*}
\RArea(A_{i,n}) \ &= \ (\lambda^{n-i} + 2w/\lambda)(\lambda^n \mu^{n-i} +
2w/\mu) - \lambda^{n-i} \lambda^n \mu^{n-i} \\
&= \ 2w\lambda^{n-1} \mu^{n-i} + 2w \lambda^{n-i} \mu^{-1} + 4w^2(\lambda
\mu)^{-1} \\
&\leq \ 2w(\lambda^{n-1} + \lambda^{n-i}\mu^{-1}) + 4w^2. 
\end{align*}
Hence the area of the horizontal part is at most 
\begin{align*}
\sum_{i=0}^{n-1} \RArea( A_{i,n}) \ &\leq \ 2w\bigl(\lambda^{n-1} + 
\lambda(\lambda^n - 1)/\mu(\lambda -1)\bigr)  + 4w^2n \\
&\leq \ 2w\bigl(\lambda^{-1} + \lambda/\mu(\lambda -1)\bigr) \lambda^n +
4w^2n. 
\end{align*}
Lastly, $4w^2 n$ is less than $\frac{4w^2}{\ln \lambda} \lambda^n$, thus 
establishing the result. 
\end{proof}

\subsection*{The bound} 
Recall that $\widetilde{X}$ contains isometric copies of $M$,
corresponding to lines in $T$. Choose two such lines $L_0$, $L_1$ which
coincide at negative heights and diverge at height $0$. Let $M_0$, $M_1$
be the corresponding copies of $M$ in $\widetilde{X}$. Let $S^i_n$ be
the subcomplex $S_n$ of $M_i$ constructed earlier (recall that the
contruction depended on the cell structure of $M_i$, which varies with
$i$). Let $B_n \subset \widetilde{X}$ be the subcomplex $S^0_n \cup
S^1_n$. It contains the two copies of $R_n$ in $M_0$ and $M_1$ (which
meet along their bottom faces), and its boundary is contained in
$\partial^+ S^0_n \cup \partial^+ S^1_n$.

Let $a$ be the minimum Riemannian area of a $2$-cell of
$\widetilde{X}$. Combining \eqref{areaRn} with Lemmas \ref{areaSn} and
\ref{areaSn2}, we have 
\begin{equation}\label{areaBn} 
 \Vol^2 (\partial B_n) \ \leq \ (2/a) \Bigl( C
\lambda^{\kappa}\bigl( 1 + (2/\ln \lambda) - (2/\ln \mu) \bigr) + D
\Bigr) \lambda^n.
\end{equation}
By \eqref{volRn2} we have 
\[\Vol^3(B_n) \ \geq \ \frac{1}{V \ln(\lambda\mu)}
\bigl(\lambda^n\bigr)^{2+\log_{\lambda}(\mu)}.\] Thus there is a constant
$E$ such that $\Vol^3(B_n) \geq E (\Vol^2(\partial B_n))^{2 +
  \log_{\lambda}(\mu)}$ for all $n$. By Remark \ref{embedded}, since
$S_n$ is embedded in $\widetilde{X}$, we have $\delta^{(2)}(x_n) \geq
E (x_n)^{2 + \log_{\lambda}(\mu)}$ for $x_n = \Vol^2(\partial
B_n)$. Lastly, it remains to show that the sequence $(x_n)$ is not too
sparse. Recall that the top $D_n$ of $\partial^+ S_n$ contains the top
face of $R_n$, and the latter has area $\lambda^n$. Thus $\Vol^2(\partial
B_n) \geq K \lambda^n$ for some constant $K$. Together with
\eqref{areaBn} this implies that the ratios $x_n/x_{n-1}$ are
bounded. According to Remark 2.1 of \cite{bbfs}, this property suffices
to conclude that $\delta^{(2)}(x) \succeq x^{2+\log_{\lambda}(\mu)}$.

\section{Proof of Theorem \ref{mainthm2}}\label{highersect}

Sections \ref{uppersec} and \ref{lowersec} established the proof of Theorem
\ref{mainthm1}. Next we consider the groups $G_{\Sigma^i A} \cong G_A
\times \Z^i$ and their $(i+2)$-dimensional Dehn functions. The following
definition is taken from \cite{bbfs}. 

\begin{definition}\label{embeddedreps} 
Let $G$ be a group of type $\mathcal{F}_{k+1}$ and geometric dimension at
most $k+1$. The $k$-dimensional Dehn function $\delta^{(k)}_G(x)$
\emph{has embedded representatives} if there is a finite aspherical
$(k+1)$-complex $X$, a sequence of embedded $(k+1)$-dimensional balls
$B_i \subset \widetilde{X}$, and a function $F(x) \simeq
\delta^{(k)}_G(x)$, such that the sequence given by $(n_i) =
(\Vol^k(\partial B_i))$ tends to infinity and is exponentially bounded,
and $\Vol^{k+1}(B_i) \geq F(n_i)$ for each $i$. 
\end{definition}

The Dehn functions $\delta^{(2)}(x)$ for the groups $G_A$ have embedded
representatives, as constructed in Section \ref{lowersec}. 
We also have the following result from \cite{bbfs}. 

\begin{proposition}\label{lowerGxZ} 
Let $G$ be a group of type $\mathcal{F}_{k+1}$ and geometric dimension at
most $k+1$. Suppose the $k$-dimensional Dehn function $\delta^{(k)}(x)$
of $G$ is equivalent to $x^s$ and has embedded representatives. Then $G
\times \Z$ has $(k+1)$-dimensional Dehn function
$\delta^{(k+1)}(x)\succeq x^{2 - 1/s}$, with embedded representatives. 
\end{proposition}

The proof of Theorem \ref{mainthm2} now proceeds exactly as in Theorem D
of \cite{bbfs}. Let $\alpha = 2 + \log_{\lambda}(\mu)$ and $s(i) =
\frac{(i + 1)\alpha - i}{i \alpha - (i -1)}$. 
We verify by induction on $i$ the following statements for $G_{\Sigma^i A}$: 
\begin{enumerate}
\item \label{st1} $\Delta^{(i+2)}(x) \leq Cx^{s(i)}$ for some constant $C
> 0$, 
\item \label{st2} $\delta^{(i+2)}(x) \succeq x^{s(i)}$, and 
\item \label{st3} $\delta^{(i+2)}(x)$ has embedded representatives. 
\end{enumerate}
The first two statements together yield the desired conclusion
$\delta^{(i+2)}(x) \simeq x^{s(i)}$. 

If $i = 0$ then \eqref{st1} and \eqref{st2} are the respective
conclusions of Sections \ref{uppersec} and \ref{lowersec}, and
\eqref{st3} holds as remarked above. 
For $ i  > 0$ note first that $s( i ) = 2 - 1/s( i  -1)$. Then
statement~\eqref{st1} holds by Theorem~\ref{product} and 
property \eqref{st1} of $G_{\Sigma^{i-1} A}$. Proposition~\ref{lowerGxZ}
implies \eqref{st2} and~\eqref{st3} by
properties~\eqref{st1}--\eqref{st3} of $G_{\Sigma^{i-1} A}$.

\section{Density of exponents}\label{densitysect}

In this section, $A$ is a $2\times 2$ matrix with integer entries. 
Denote the trace and determinant of $A$ by $t$ and $d$ respectively. Note
that the characteristic polynomial of $A$ is given by $p(x) = x^{2} -{t}
x + {d}$, and the eigenvalues are $\lambda \; = \; \frac{{t} +
\sqrt{{t}^{2} - 4{d}}}{2}$ and $\mu \; = \; \frac{{t} - \sqrt{{t}^{2} -
4{d}}}{2}$. The next lemma shows that under certain conditions, the
leading eigenvalue can be roughly approximated by the trace. 

\begin{lemma} 
\label{lambdatau} 
If ${t} \geq 4$ and $t \geq d \geq 0$ then $\lambda, \mu \in \R$ and \/
${t} -4 \; \leq \; \lambda \; \leq \; {t}$. 
\end{lemma}

\begin{proof}
First, $t \geq 4$ and $t \geq d$ imply that $t^2 \geq 4d$, and therefore
$\lambda, \mu \in \R$. Next, $\lambda$ is the average of $t$ and
$\sqrt{t^2 - 4d}$, and so $\sqrt{t^2 - 4d} \leq \lambda \leq t$. It
remains to show that $t-4 \leq \sqrt{t^2 - 4d}$. Note that
$\sqrt{t^2 - 4t}$ is the geometric mean of $t-4$ and $t$, and so it
lies between $t-4$ and $t$. Since $t \geq d$, we now have $t-4 \leq
\sqrt{t^2 - 4t} \leq \sqrt{t^2 - 4d}$, as needed. 
\end{proof}

\begin{lemma}\label{densitylemma} 
The function $f(x,y) = \log_x(y)$ maps the set 
\[ S \ = \ \{ \, (t,d) \in \N \times \N \mid 2 \leq d \leq t-4 \, \}\] 
onto a dense subset of \/ $(0,1)$. 
\end{lemma}

\begin{proof}
Given $\epsilon > 0$, fix an integer $t > e^{2/\epsilon}$. We will
show that the points $(t,2)$, $(t,3)$, \dots, $(t,t-4)$ map to an
$\epsilon$-dense subset of $(0,1)$. 

Fixing $x = t$, the function $f(t, \, \cdot \, )$ maps $[1,t]$
homeomorphically onto $[0,1]$, and maps $[2,t]$ onto an interval
containing $[\epsilon, 1]$, by the choice of $t$. Since $f_y =
\frac{1}{y \ln(x)}$, we have $\abs{f_y(t,y)} \leq \frac{1}{2 \ln(t)} <
\epsilon/4$ for all $y \geq 2$, again by the choice of $t$. Therefore 
\[ \abs{f(t,d) - f(t,d+1)} \ < \ \epsilon/4\]
for all integers $d \geq 2$. Thus the image of the set $\{(t,2)$,
$(t,3)$, \dots, $(t,t)\}$ is $\epsilon/4$-dense in (and includes the
endpoints of) an interval
containing $[\epsilon,1]$. Omitting the last four points, the remaining
set is $\epsilon$-dense in $(0,1)$. 
\end{proof}

Now we can prove the main result of this section. 
\begin{proposition}[Density]\label{densityprop}
Given $\alpha \in (1,2)$ and $\epsilon > 0$, there is a matrix $A \in
M_{2}(\Z)$ with determinant ${d} \geq 2$ and eigenvalues $\lambda$, $\mu$
with $\lambda > 1 > \mu$ such that \/ $\abs{\bigl(2 +
\log_{\lambda}({\mu})\bigr) -\alpha} \ < \ \epsilon$. 
\end{proposition}

\begin{proof}
Given integers $t$ and $d$, the matrix
\[ A(t,d) \ = \ \begin{pmatrix} { \ t }  & - {d} \ \\
\ 1 & 0 \end{pmatrix} \ \in \ M_2(\Z)\]
has trace $t$ and determinant $d$ (and eigenvalues $\lambda, \mu$). 
Note also that $\lambda \mu = d$ implies that $2 + \log_{\lambda}(\mu) =
1 + \log_{\lambda}(d)$. Thus we need to choose $t$ and $d$ so that
$\log_{\lambda}(d)$ is within $\epsilon$ of $\alpha - 1$. 

First, choose a number $T$ such that 
\begin{equation}\label{Tbound} 
\frac{4}{(t-4)\ln(t-4)} \ \leq \ \epsilon/2
\end{equation}
for all $t \geq T$. 

Next, apply Lemma \ref{densitylemma} to obtain $t$ and $d$ such that
$\abs{\log_{t}(d) - (\alpha - 1)} < \epsilon/2$ and $2 \leq d \leq
t-4$. We may assume in addition that $t \geq T$, since only finitely many
points of $S$ violate this condition, and omitting these from $S$ does
not affect the conclusion of the lemma. By Lemma \ref{lambdatau} we have 
\begin{equation}\label{big-ineq}
2 \ \leq \ d \ \leq \ t-4 \ \leq \ \lambda \ \leq \ t.
\end{equation}

Note that $f(x,y) = \log_x(y)$ has partial derivative 
$f_x = \frac{-\ln(y)}{x \ln(x) \ln(x)}$. Along the segment $\{(x,y) \mid t-4
\leq x \leq t, \ y=d\}$ we have 
\[ \abs{f_x} \ \leq \ \frac{\ln(d)}{(t-4)\ln(t-4) \ln(t-4)} \ \leq \
\frac{1}{(t-4)\ln(t-4)}.\]
This implies (with \eqref{Tbound}) that 
\[ \abs{\log_{t-4}(d) - \log_t(d)} \ \leq \ \frac{4}{(t-4)\ln(t-4)} \
\leq \ \epsilon/2.\] 
Now, since $\lambda$ is between $t-4$ and $t$, we have 
\[ \abs{\log_{\lambda}(d) - \log_t(d)} \ \leq \ \epsilon/2,\] 
and hence $\log_{\lambda}(d)$ is within $\epsilon$ of $\alpha - 1$. 

Lastly, the inequality $\mu < 1$ reduces to $d < t-1$, which holds by
\eqref{big-ineq}. The inequality $\lambda > 1$ is clear since $t \geq 2$. 
\end{proof}


\def\cprime{$'$}

\end{document}